%% file: Boutillier-Raschel-21.tex
\documentclass[a4paper,11pt,reqno]{amsart}
\usepackage[english]{babel}
\usepackage[utf8]{inputenc} 
\usepackage{cmbright}
\usepackage[T1]{fontenc}
\usepackage{amsmath}
\usepackage{amsfonts}
\usepackage{amssymb}
\usepackage{amsthm,eepic}
\usepackage{mathrsfs}
\usepackage{enumitem,epic}
\usepackage{graphicx}
\usepackage{footnote,color}
\usepackage{hyperref}
\usepackage{comment}
\usepackage{xcolor}
\hypersetup{colorlinks=true,    linkcolor=blue,
citecolor=red,       filecolor=BrickRed,       urlcolor=darkgreen
}
\newmuskip\pFqmuskip

\newcommand*\pFq[6][8]{%
  \begingroup 
  \pFqmuskip=#1mu\relax
  \mathcode`\,=\string"8000
  \begingroup\lccode`\~=`\,
  \lowercase{\endgroup\let~}\pFqcomma
  {}_{#2}F_{#3}{\left[\genfrac..{0pt}{}{#4}{#5};#6\right]}%
  \endgroup
}
\newcommand{\pFqcomma}{\mskip\pFqmuskip}
\newcommand{\Vs}{\mathsf{V}}
\newcommand{\Gs}{\mathsf{G}}
\newcommand{\Es}{\mathsf{E}}
\newcommand{\GR}{\Gs^{\diamond}}
\newcommand{\CC}{{\mathbb C}}
\newcommand{\xb}{x}
\newcommand{\yb}{y}
\newcommand{\ps}{\hat{\mathsf{r}}}
\newcommand{\pps}{\check{\mathsf{r}}}
\newcommand{\Arm}{\mathrm{A}}
\DeclareMathOperator{\Dc}{Dc}
\DeclareMathOperator{\dc}{dc}
\DeclareMathOperator{\dn}{dn}
\DeclareMathOperator{\ns}{ns}
\DeclareMathOperator{\cn}{cn}
\DeclareMathOperator{\sn}{sn}
\DeclareMathOperator{\nc}{nc}
\DeclareMathOperator{\Id}{Id}
\newcommand{\ud}{\mathrm{d}}
\newcommand{\TT}{{\mathbb T}}
\newcommand{\ZZ}{{\mathbb Z}}
\newcommand{\RR}{{\mathbb R}}

\renewcommand{\geq}{\geqslant}
\renewcommand{\leq}{\leqslant}
\renewcommand{\epsilon}{\varepsilon}
\theoremstyle{plain}
\newtheorem{thm}{Theorem}

\newtheorem{prop}[thm]{Proposition}
\newtheorem{lem}[thm]{Lemma}
\theoremstyle{definition}

\theoremstyle{remark}

\definecolor{redd}{rgb}{0.9,0,0}

\usepackage{color}
\definecolor{darkgreen}{rgb}{0,0.4,0}
\definecolor{MyDarkBlue}{rgb}{0,0.08,0.50}
\definecolor{BrickRed}{rgb}{0.65,0.08,0}

\let\sc\relax
\DeclareMathOperator{\sc}{sc}
\DeclareMathOperator{\expo}{\mathsf{e}}

\title{Martin boundary of killed random walks on isoradial graphs}

\author[C\'edric Boutillier and Kilian Raschel]{C\'edric Boutillier and Kilian Raschel, with an appendix by Alin Bostan}

\thanks{A.\ Bostan: Inria, Universit\'e Paris-Saclay, 1 rue Honor\'e d’Estienne d’Orves, 91120 Palaiseau, France}

\thanks{C.\ Boutillier: Sorbonne Université, CNRS, Laboratoire de Probabilit\'es, Statistique et
Mod\'elisation (UMR 8001), 4 Place Jussieu, F-75005 Paris, France \& Institut Universitaire de France}

\thanks{K.\ Raschel: CNRS \& Institut Denis Poisson (UMR 7013), Universit\'e de Tours et Universit\'e d'Orl\'eans, Parc de Grandmont, 37200 Tours, France}
\thanks{This project has received funding from the European Research Council (ERC) under the European Union's Horizon 2020 research and innovation programme under the Grant Agreement No.\ 759702 and from the ANR DIMERS (ANR-18-CE40-0033).}

\email{alin.bostan@inria.fr}
\email{cedric.boutillier@sorbonne-universite.fr}
\email{raschel@math.cnrs.fr}


\date{\today}

\begin{document}

\begin{abstract}
We consider killed planar random walks on isoradial graphs. Contrary to the lattice case, isoradial graphs are not translation invariant, do not admit any group structure and are spatially non-homogeneous. Despite these crucial differences, we compute the asymptotics of the Martin kernel, deduce the Martin boundary and show that it is minimal.
Similar results on the grid $\mathbb Z^d$ are derived in a celebrated work of Ney and Spitzer.
\end{abstract}

\maketitle
\setcounter{tocdepth}{1}
\tableofcontents

\section{Introduction and results}

\subsection{Ney and Spitzer theorems}
\label{subsec:NS}

Let $(Z_n)_{n\geq0}$ be a random walk on $\mathbb Z^d$ with finite support and
non-zero drift. In the celebrated paper \cite{NeSp-66}, Ney and Spitzer derive
the asymptotics of the Green function
\begin{equation}
\label{eq:Green_function_1}
     G(x,y)
     =\sum_{n=0}^\infty \mathbb P_x[Z_n=y]
\end{equation}
as $\vert y\vert\to\infty$ in any given direction $\frac{y}{\vert y\vert}\to
\ps\in\mathbb S^{d-1}$, see \cite[Thm~2.2]{NeSp-66}. Introduce the one-step transition probabilities generating function (with
$\zeta\cdot y$ denoting the standard scalar product of $\zeta$ and $y$ in $\mathbb R^d$)
\begin{equation}
\label{eq:def_phi}
     \phi(\zeta)=\sum_{y\in\mathbb Z^d} p(0,y)e^{\zeta\cdot y},
\end{equation}
defined in terms of the transition kernel $p(x,y)=\mathbb{P}_x[Z_1=y]$ of the random
walk. Ney and Spitzer deduce from their asymptotics of~\eqref{eq:Green_function_1} that
\begin{equation}
\label{eq:asymptotic_Martin_kernel_NS}
\lim_{\substack{\vert y\vert\to\infty\\y/\vert y\vert\to\ps}}\frac{G(x_1,y)}{G(x_0,y)}=e^{\zeta(\ps)\cdot (x_1-x_0)},
\end{equation}
where the mapping $\ps\mapsto\zeta(\ps)$ is proved to define a homeomorphism between the ambient
 sphere $\mathbb S^{d-1}$ and the set
\begin{equation}
\label{eq:def_partial_D}
     \partial D=\{\zeta\in\mathbb R^d: \phi(\zeta)=1\},
\end{equation}
see \cite[Cor.~1.3]{NeSp-66} and~\cite{He-63}; see also~\eqref{eq:grad_p} for an expression of the above homeomorphism. Figure \ref{fig:example_curve} contains a few examples of sets $\partial D$ in \eqref{eq:def_partial_D}.

Introduce now the discrete Laplacian $\Delta=\Id-p$, i.e., for $x\in\mathbb Z^d$,
\begin{equation*}
     \Delta f(x) = \sum_{y\in\mathbb{Z}^d}p(x,y)(f(x)-f(y)).
\end{equation*}
From a potential theory viewpoint it is obvious that the exponential functions
$f(x)=e^{\zeta\cdot x}$ are discrete harmonic (meaning $\Delta f=0$) if and only if $\zeta\in\partial D$.
It is further known~\cite{ChDe-60,DoSnWi-60} that they are (positive)
\emph{minimal}, in the sense that if $h$ is harmonic and satisfies $0\leq h\leq f$ on $\mathbb Z^d$, then $h=c f$ for some $0\leq c\leq1$.

A corollary of Ney and Spitzer
result~\eqref{eq:asymptotic_Martin_kernel_NS} is that the Martin
compactification of the random walk coincides with the geometric
compactification of $\mathbb Z^d$ by the sphere $\mathbb S^{d-1}$ at infinity.
More precisely, any harmonic function $f$ for $(Z_n)_{n\geq0}$ may be written as
$f(x)=\int_{\partial D} e^{\zeta\cdot x} \mu(\ud\zeta)$, where $\mu$ is a positive Radon
measure on $\partial D$, see \cite[Rem.~a]{NeSp-66}. For a general introduction
on Martin boundary see \cite{Ma-41,Do-59,Hu-60,Sa-97,Wo-00}. 

Let us mention a few strongly related results. First, the original Ney and Spitzer
result \eqref{eq:asymptotic_Martin_kernel_NS} does not assume the boundedness of
the support (it is only assumed that every point of $\partial D$ has a
neighborhood in which $\phi$ in~\eqref{eq:def_phi} is finite), but throughout this article we shall
only consider random walks with bounded increments. In absence of drift and if
$d\geq 3$ (so as to have a transient behavior), the Martin boundary is reduced to a
single point, see \cite[Sec.~26]{Sp-64} and \cite[(25.11) and (25.12)]{Wo-00}
for exact statements and proofs.
Let us also point out the concept of $t$-Martin boundary
(which probabilistically corresponds to a random walk with killing), which
involves the Green function
\begin{equation}
\label{eq:Green_function_t}
G(x,y\vert \tfrac{1}{t})=\sum_{n=0}^\infty \mathbb{P}_x[Z_n=y]\frac{1}{t^n}.
\end{equation}
Obviously $G(x,y\vert 1)=G(x,y)$ in our notation \eqref{eq:Green_function_1}. This $t$-Martin boundary is either empty (when $t<\rho$, where $\rho$ is the spectral radius of the transition kernel), reduced to one point ($t=\rho$) or homeomorphic to $\mathbb S^{d-1}$ (when $t>\rho$), see Woess \cite[Sec.~25.B]{Wo-00}. The ratio $G(x_1,y\vert \tfrac{1}{t})/G(x_0,y\vert \tfrac{1}{t})$ converges to the same quantity as in \eqref{eq:asymptotic_Martin_kernel_NS}, but with $\zeta(\ps)\in\partial D_t$, where
\begin{equation}
\label{eq:def_partial_D_t}
     \partial D_t=\{\zeta\in\mathbb R^d: \phi(\zeta)=t\}.
\end{equation}
We finally mention the work \cite{Du-20}, in which Dussaule obtains the Martin boundary of thickenings of $\mathbb Z^d$ (i.e., Cartesian products of $\mathbb Z^d$ by a finite set). In the reference \cite{Ba-88} Babillot considers non-lattice random walks and obtains the Martin boundary of thickenings of $\mathbb R^d$.
\begin{figure}[ht]
\includegraphics[width=0.4\textwidth]{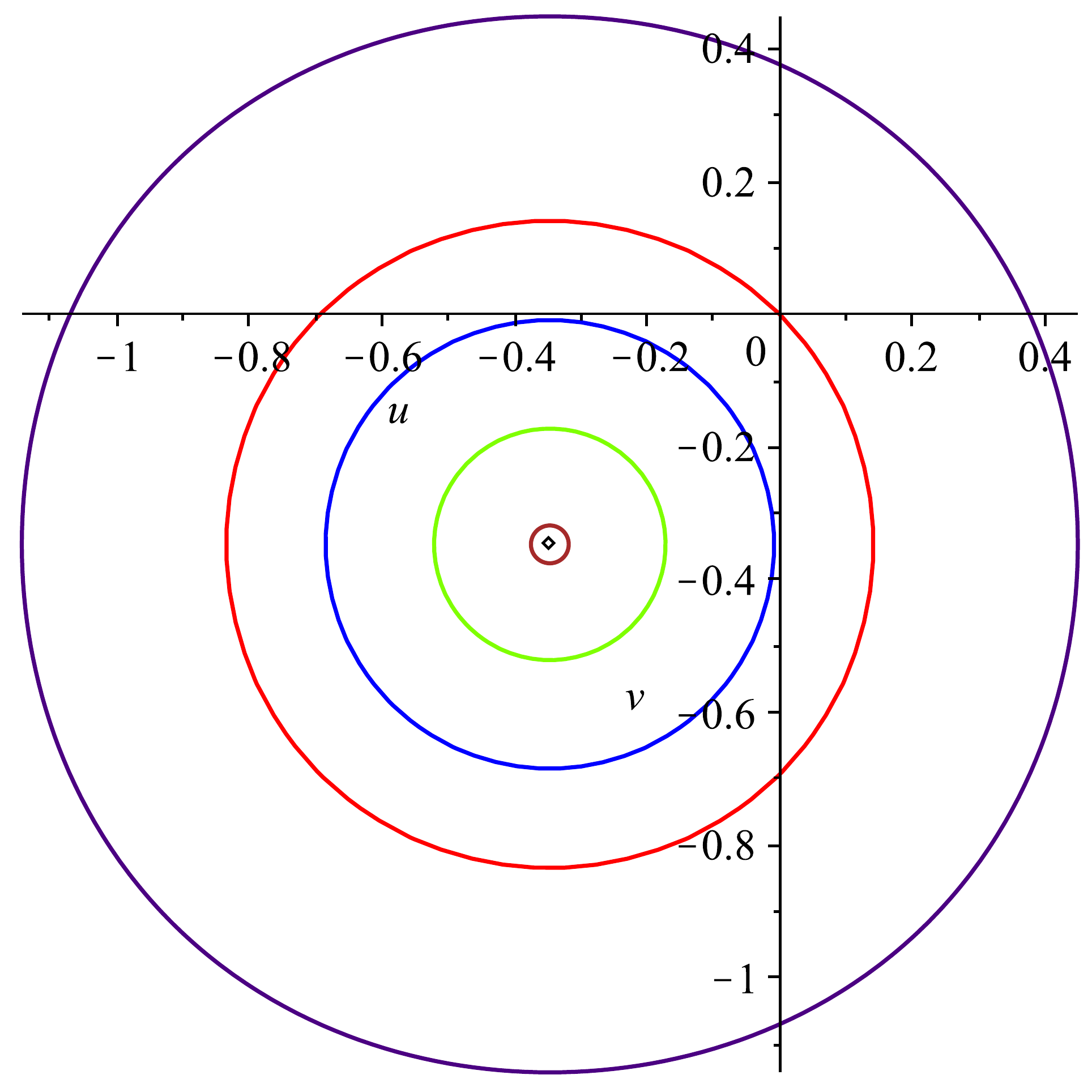}
\caption{The curve $\phi(\zeta_1,\zeta_2)=t$, for
$\phi(\zeta_1,\zeta_2)=\frac{e^{\zeta_1}}{3}+\frac{e^{\zeta_2}}{3}+\frac{e^{-\zeta_1}}{6}+\frac{e^{-\zeta_2}}{6}$
and $t=2\sqrt{2}/3\approx 0.9428,0.943,0.95,0.97,1,1.1$. 
It corresponds to a
nearest neighbor walk in $\mathbb{Z}^2$ with jump probability $\frac{1}{3}$ to
the North and East, and $\frac{1}{6}$ to the South and West. The spectral radius of
the transition kernel is $2\sqrt{2}/3$.}
\label{fig:example_curve}
\end{figure}

\subsection{Spatial inhomogeneity}
A crucial ingredient in the proof of all above results is the spatial homogeneity (or invariance by translation) of $\mathbb Z^d$ and of the random walk, meaning that $p(x,y)=p(0,y-x)$ for all $x,y\in\mathbb Z^d$.

The perturbation of the homogeneity at a finite number of sites of the grid $\mathbb Z^d$ should not affect the structure of the Martin boundary, but already impacts on the expressions for harmonic functions. This is shown by Kurkova and Malyshev \cite[Thm~2.1]{KuMa-98} in the case of planar random walks with jumps to the nearest neighbors; they obtain the asymptotics of the Green function and show that the Martin boundary is homeomorphic to the circle $\mathbb S^1$. On the other hand, in the case of a homogeneity perturbed at an infinite number of points, computing the Green function asymptotics or deriving the Martin boundary seems too difficult in general (even if some upper and lower bounds may exist for the transition probabilities or for the Green functions, see \cite{Mu-06}). 

It is therefore natural to introduce a structure behind the inhomogeneities. A class of such models is composed of random walks killed or reflected when they exit a domain (inhomogeneities then appear on the boundary of the domain). In this context, Kurkova and Malyshev \cite{KuMa-98}, Ignatiouk-Robert \cite{IR-08,IR-10a,IR-10b}, Ignatiouk-Robert and Loree \cite{IRLo-10} obtain the asymptotics of (ratios of) Green functions as in \eqref{eq:asymptotic_Martin_kernel_NS} and derive the Martin boundary when the domain is a quarter plane $\mathbb Z_+^2$ or a half-space $\mathbb Z_+\times \mathbb Z^{d-1}$. See also \cite{Go-04}.

\subsection{The case of isoradial graphs}
In this article we explore another class of models having  inhomogeneities in
infinite number, namely, killed random walks on isoradial graphs. By definition,
isoradial graphs (see Figure~\ref{fig:Iso0} for a few examples) are planar
embedded graphs such that all faces are inscribable in a circle of radius $1$;
more details will be provided in Section~\ref{sec:iso_expo}. So the irregularities of
such graphs are structured, but as it is illustrated by looking at the
pictures, the behavior may be quite wild. In particular, with the exception of a
few simple particular cases (the square lattice $\mathbb Z^2$ or the
triangular lattice), an isoradial graph is not a linear transform of a two-dimensional lattice, is not translation-invariant and can be actually highly
non-homogeneous. Introduced by Duffin \cite{Du-68}, Mercat \cite{Me-01} and Kenyon 
\cite{Ke-02}, these graphs have become
particularly popular, as they are well suited for discrete complex analysis~\cite{ChSm-11}.
Moreover, because the star-triangle transformation (see Figure \ref{fig:triangle_etoile}), which plays an important
role in integrable statistical mechanics, preserves isoradiality, many models of
statistical mechanics are now studied \cite{BodT-12} on isoradial graphs: dimer
model~\cite{Ke-02,deTil:iso}, two-dimensional Ising model~\cite{BodTRa-19}, spanning trees,
bond percolation~\cite{Grimmett-Mano}, random cluster model~\cite{BDCS15,Du-Li-Mano}, etc.

Despite their intrinsic irregularity and the absence of group structure,
isoradial graphs form a broad class of graphs on which one can 
define an exponential function in a uniform way;
this is a first step towards a Ney and Spitzer theorem.
Such an exponential function is only proved to exist when the isoradial graph
$\Gs$ is equipped with certain natural conductances, involving trigonometric or
elliptic functions and geometric angles of the embedding of the graph. The
existence is obtained by Kenyon
\cite{Ke-02} for critical (trigonometric) weights and is extended to the
non-critical (elliptic) case by de
Tili\`ere and the two present authors~\cite{BodTRa-17}. Whereas the classical
exponential function between any two points $x,y\in\mathbb Z^d$ may be written
as $e^{\zeta\cdot (y-x)}$ with a parameter $\zeta\in\mathbb C$, see
\eqref{eq:asymptotic_Martin_kernel_NS}, the exponential function on $\Gs$ takes
the form 
\begin{equation}
\label{eq:intro_expo}
     \expo_{(x,y)}(u\vert k),
\end{equation}
where $x,y\in\Gs$, $k\in(0,1)$ is an
elliptic modulus naturally attached to the model and $u$ lies on the torus $\TT(k)=\mathbb C/(4K(k)\mathbb
Z+4iK'(k)\mathbb Z)$, $K$ and $K'$ denoting the elliptic integrals 
of the first kind and its complementary, respectively, see 
Section~\ref{sec:expo_def} and \cite[Sec.~3.3]{BodTRa-17}. In the same
manner that $f(x)=e^{\zeta\cdot x}$ is harmonic on $\mathbb Z^d$ for $\zeta\in\partial
D$, the exponential function $\expo_{(x,y)}(u\vert k)$ is discrete harmonic as a function of $x$ ($y$ being fixed), in
the sense that its massive Laplacian 
\begin{equation}
\label{eq:Laplacian_operator}
     \Delta f(x)=\sum_{z\sim x} \rho(\overline{\theta}_{xz}\vert k)(f(x)-f(z))+m^{2}(x\vert k)f(x)
\end{equation}
is zero by \cite[Prop.~11]{BodTRa-17}, the conductances $\rho$ and (squared) masses $m^{2}$ being defined in \eqref{eq:conductances} and \eqref{eq:mass}. The work \cite{BodTRa-17} also provides exact and asymptotic expressions for the Green function $G(x,y)$ of the killed random walk described by \eqref{eq:Laplacian_operator}, see in particular \cite[Thm~14]{BodTRa-17}.

\subsection{Main results}
Building on the results of \cite{BodTRa-17}, our main theorem is the asymptotics of the Martin kernel
    \begin{equation}
\label{eq:asymptotic_Martin_kernel_intro}
     \lim_{y \to\infty}\frac{G(x_1,y)}{G(x_0,y)}=\expo_{(x_0,x_1)}(u_0\vert k),
\end{equation} 
where the exponential function $\expo_{(x_0,x_1)}$ in \eqref{eq:intro_expo} is evaluated at a point $u_0\in2iK'+\mathbb R$ which depends on the way that $y$ goes to infinity in the graph. See Theorem~\ref{thm:main-1} for the precise statement. Let us present the main features of \eqref{eq:asymptotic_Martin_kernel_intro}:
\begin{itemize}
     \item First, it gives the limit of the ratio of the Green function in any asymptotic direction in the graph. It also proves that the Martin boundary is homeomorphic to the sphere $\mathbb S^1$, providing the announced generalization of Ney and Spitzer result \eqref{eq:asymptotic_Martin_kernel_NS} on isoradial graphs. The latter result follows from showing that $u_0$ describes a circle in the torus, as the asymptotic direction varies along all possible directions in the graph (see our Theorem~\ref{thm:main-2}).
     \item Theorem~\ref{thm:main-3} shows that for all $u_0\in2iK'+\mathbb R$ and $y\in\Gs$, $\expo_{(x,y)}(u_0\vert k)$ defines a minimal positive harmonic function (see Section \ref{subsec:NS} for the definition), thereby proving that the Martin boundary is minimal. Recall that positive harmonic functions are particularly important in potential theory, as they allow, via the Doob transformation, to construct conditioned processes (for example going to infinity along a given direction in the graph).
     \item In the few particular cases which are both isoradial graphs and lattices (namely, the square and triangular lattices), we show that our results coincide with the classical Ney and Spitzer theorem and we unify the two points of view.     \item Beyond the square and triangular lattices, our results in particular
       apply to periodic isoradial graphs, which are constructed as
       lifts to the plane of a fundamental pattern on the torus. We prove that
       some tools developed
       for their study (characteristic polynomial, amoeba, quantity $u_0$ in
       \eqref{eq:asymptotic_Martin_kernel_intro}, etc.)\ naturally correspond
       with objects in $\mathbb Z^d$ (generating function of the transition
       probabilities, set $\partial D_t$ in \eqref{eq:def_partial_D_t}, mapping
       $\zeta(\ps)$ in \eqref{eq:asymptotic_Martin_kernel_NS}, etc.).
\end{itemize}
The paper is organized as follows. Section~\ref{sec:iso_expo} presents some
crucial definitions and properties of isoradial graphs, Green functions and
discrete exponential functions. Section~\ref{sec:main_result} contains the
statements of the main results and their proofs.
After gathering and reinterpreting existing results on general periodic planar
graphs,
Section~\ref{sec:periodic_case} presents a refined study of the case of periodic isoradial graphs. Section~\ref{sec:minimal_harmonic_functions} focusses on minimal harmonic functions.

\subsection*{Acknowledgments} We warmly thank B\'eatrice de Tili\`ere for
interesting discussions.
We also thank an anonymous referee for his/her reading and comments.

\section{Isoradial graphs, Green functions and exponential functions}
\label{sec:iso_expo}

\subsection{Isoradial graphs}
\label{sec:iso_def}

Isoradial graphs, named after~\cite{Ke-02}, are defined as follows\footnote{The
presentation of Section \ref{sec:iso_def} follows that of
\cite[Sec.~2]{BodTRa-17}.}.  An embedded planar graph $\Gs=(\Vs,\Es)$ is
\emph{isoradial} if all internal faces are inscribable in a circle, with all circles
having the same radius (fixed for example to $1$), and such that all circumcenters are in the interior of
the faces. Examples are provided in Figures~\ref{fig:Iso0} and \ref{fig:waves}. 

Let $\Gs$ be an infinite, isoradial graph,
  whose bounded faces fill the whole plane.
Then the dual graph $\Gs^*$,
 embedded by placing dual vertices at the circumcenters of the corresponding
 faces, is also isoradial.

\begin{figure}[ht!]
    \begin{tabular}{cc}
      \includegraphics[width=6cm,height=7.2cm]{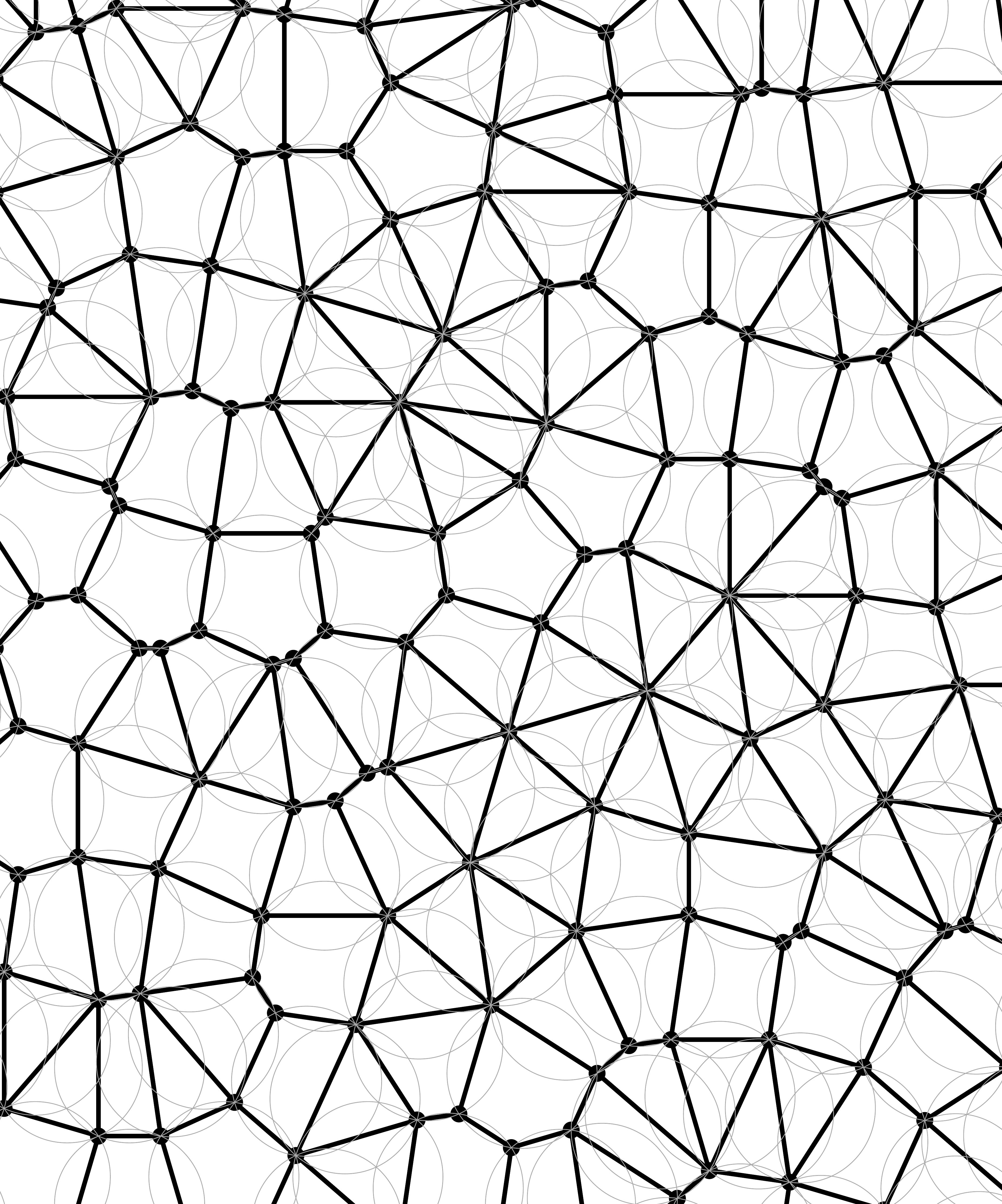} &
      \includegraphics[width=6cm,height=7.2cm]{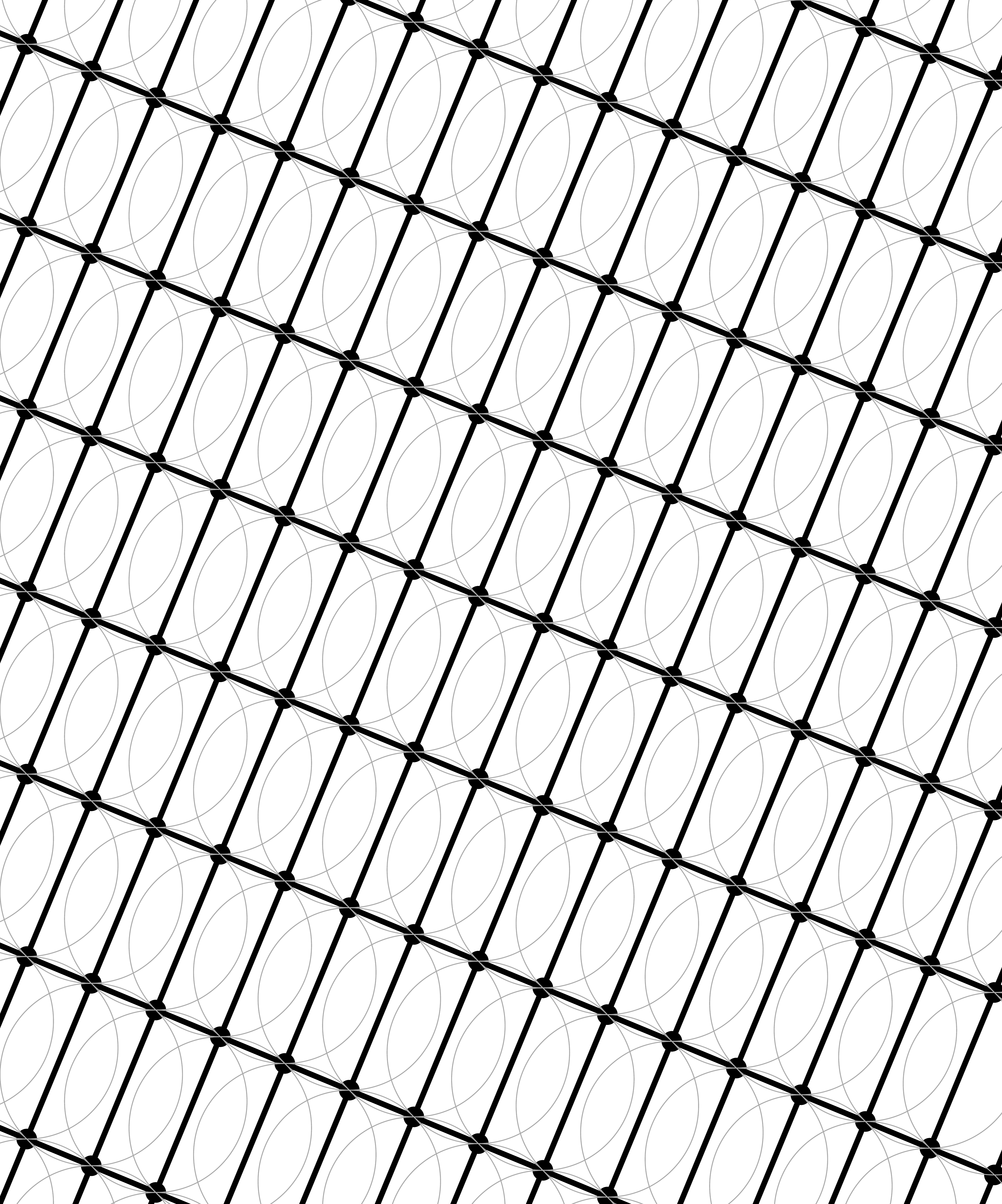} \\
      \includegraphics[width=6cm,height=7.2cm]{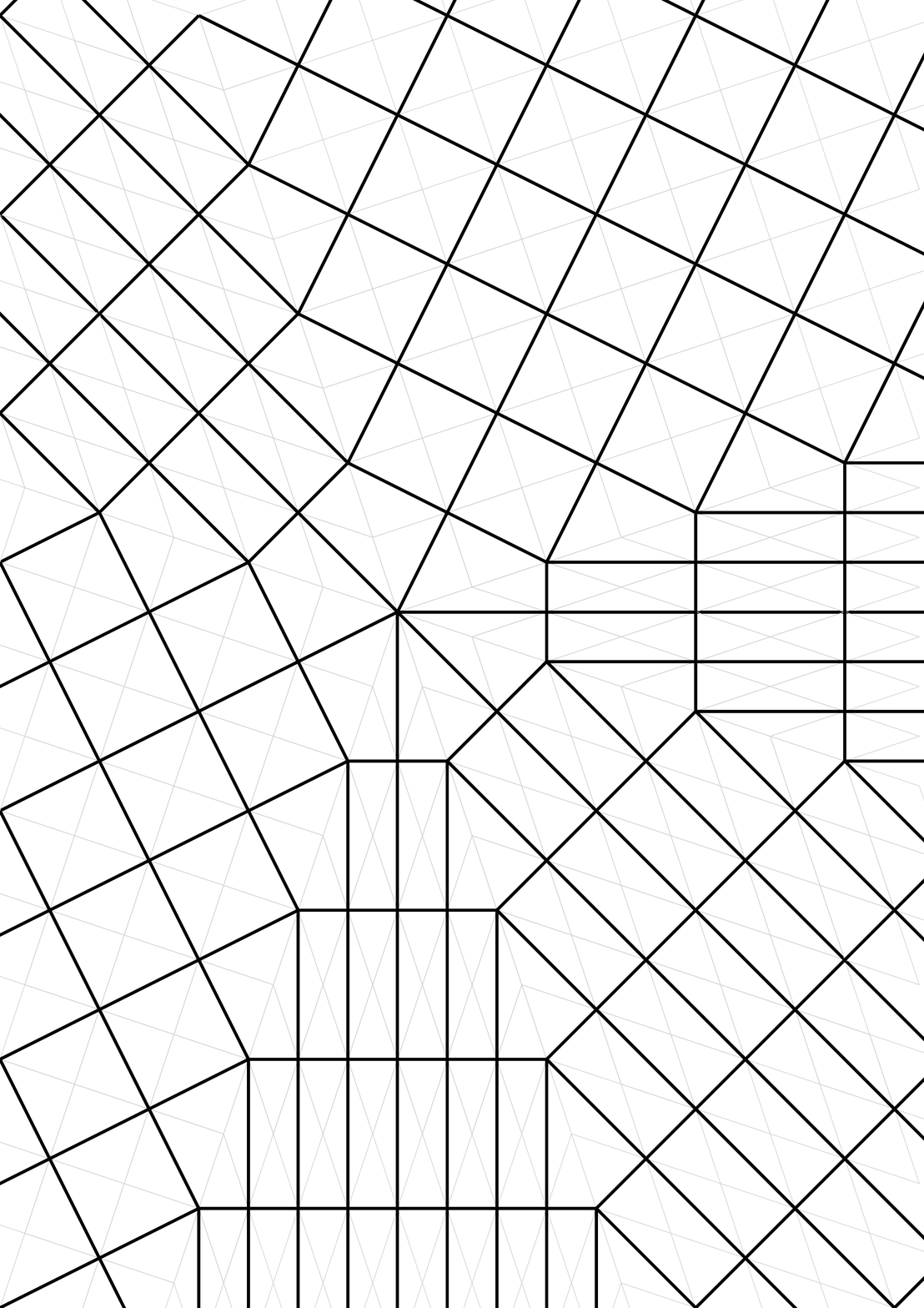} &
      \includegraphics[width=6cm,height=7.2cm]{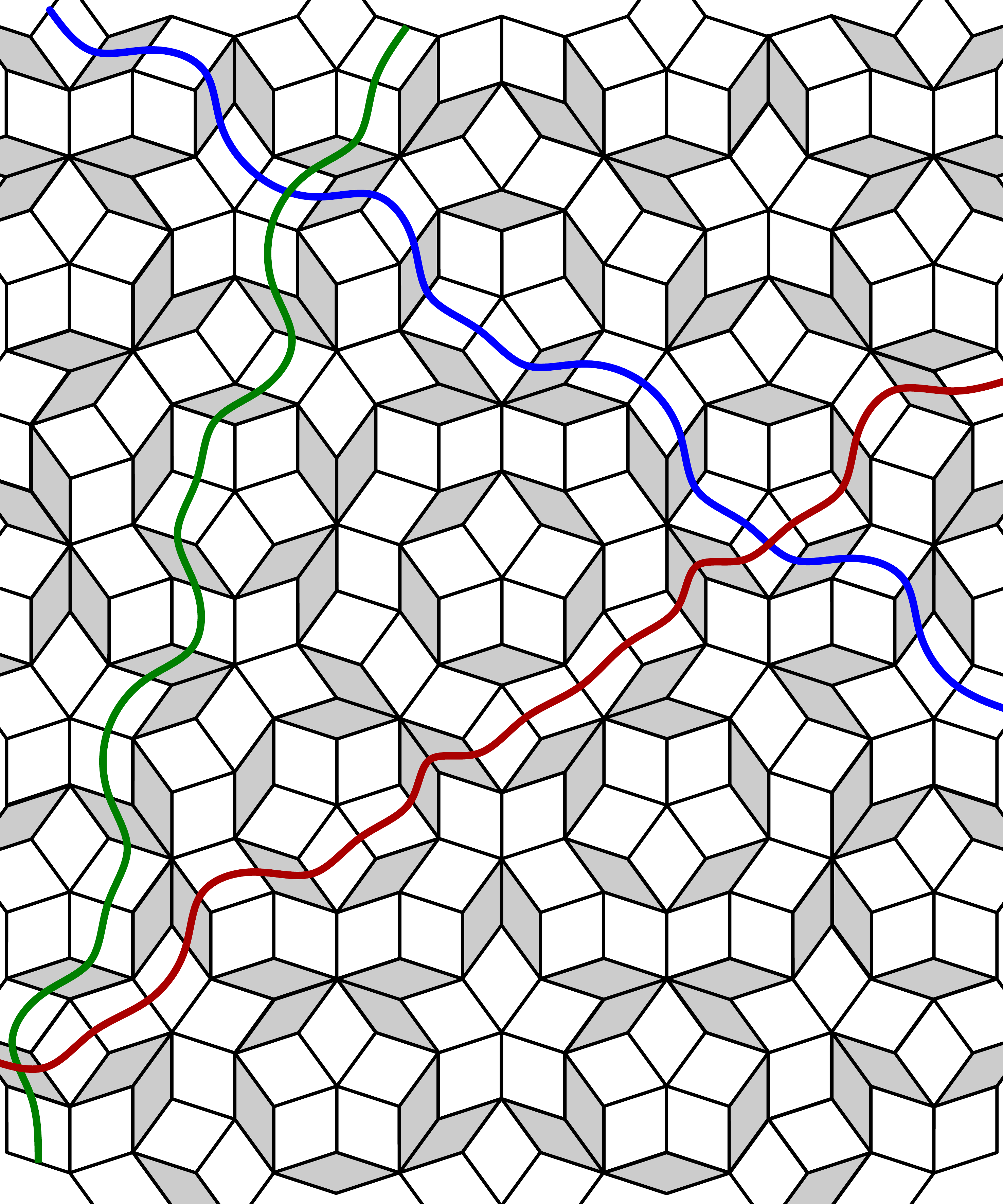}
    \end{tabular}
    \caption{Upper left: piece of an infinite graph $\Gs$ isoradially embedded in the plane
      with the circumcircles of the faces.
      Upper right: a periodic isoradial embedding of the square lattice.
    Lower left: an isoradial graph with multiple cone-like regions (see Section \ref{sec:cones} for
    a related discussion).
    Lower right: a piece of
    the Penrose tiling with rhombi, which can be used as the diamond
  graph of an isoradial graph, and three train-tracks, as bi-infinite paths in
  the dual.}
    \label{fig:Iso0}
\end{figure}

\begin{figure}
\begin{center}
\resizebox{0.18\textwidth}{!}{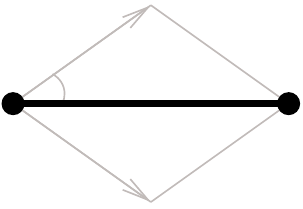}
\end{center}
\caption{An edge $e$ of $\Gs$ is the diagonal of a rhombus of $\GR$, defining the angle $\overline{\theta}_e$ and the rhombus vectors
$e^{i\overline{\alpha}_e}$ and $e^{i\overline{\beta}_e}$. One has $\frac{\overline{\beta}_e-\overline{\alpha}_e}{2}=\overline{\theta}_e$.}
\label{fig:rhombus_angle}
\end{figure}

The \emph{diamond graph}, denoted $\GR$, is constructed from
an isoradial graph $\Gs$ and its dual $\Gs^*$.
Vertices of $\GR$ are those of~$\Gs$ and those of $\Gs^*$. A dual vertex of $\Gs^*$ is joined to all primal
vertices on the boundary of the corresponding face. Since edges of the diamond graph $\GR$ are radii of circles, 
they all have length $1$ and can be assigned a direction $\pm e^{i\overline{\alpha}}$. Note that faces of $\GR$ are 
side-length $1$ rhombi.

A \emph{train-track} is a bi-infinite path on the dual of the diamond graph
$(\GR)^*$ which does not turn: when entering in a rhombus, it exits through the
opposite edge. See Figure~\ref{fig:Iso0}. A train-track $T$ thus crosses edges with the same direction
$\pm e^{i\overline{\alpha}}$. We then say that the angle $\overline{\alpha}$ is
\emph{associated} to the train-track. This angle is a priori well defined only
modulo $\pi$. This ambiguity can be lifted by considering an orientation
for the train-track and choosing a convention (i.e., the vector
$e^{i\overline{\alpha}}$ crosses the oriented train-track from left to right),
and then, the angle associated to the train-track oriented in the other direction is
$\overline{\alpha}+\pi$.

A \emph{minimal path} between two vertices $x$ and $y$
is a shortest path in $\GR$ between these vertices.
It is not unique, but the
set of unit steps used is the same for all minimal paths.

Using the diamond graph, angles can naturally be assigned to edges of the graph~$\Gs$ as follows.
Every edge $e$ of $\Gs$ is the diagonal of exactly one
rhombus of $\GR$, and we let $\overline{\theta}_e$ be the half-angle at
the vertex it has in common with $e$, see Figure~\ref{fig:rhombus_angle}.
Note that we have $\overline{\theta}_e\in(0,\frac{\pi}{2})$, because circumcircles are assumed to be in the interior of the faces.
From now on, we actually ask more and suppose that there exists $\epsilon>0$ 
such that $\overline{\theta}_e\in(\epsilon,\frac{\pi}{2}-\epsilon)$.
We also assign two rhombus vectors to the edge $e$, denoted 
$e^{i\overline{\alpha}_e}$ and $e^{i\overline{\beta}_e}$, see Figure~\ref{fig:rhombus_angle}, and
we assume that $\overline{\alpha}_e,\overline{\beta}_e$ satisfy
$\frac{\overline{\beta}_e-\overline{\alpha}_e}{2}=\overline{\theta}_e$.

An isoradial graph $\Gs$ is said to be \emph{quasicrystalline} if the number $d$ of
possible directions assigned to edges of its diamond graph $\GR$ is finite.

\subsection{Monotone surfaces and asymptotically flat isoradial graphs}
\label{subsec:monotone}

In the quasicrystalline case, the diamond graph can be seen as the
projection of a monotone surface in $\mathbb{Z}^d$ onto the plane, and every edge
of $\GR$ is the projection of a unit vector of $\mathbb{Z}^d$, see \cite{dB-81,BoSu-08}. 
Vertices (resp.\ faces) of $\Gs$ correspond then to even (resp.\ odd) vertices of
$\mathbb{Z}^d$ on that surface.
For example, in the particular case when $d=3$, the diamond graph appears
naturally to the eye as the projection in the plane of a landscape made of unit
cubes, see for example Figure~\ref{fig:waves}. Also, the Penrose tiling (see
Figure~\ref{fig:Iso0},
lower right) is the projection of a monotone surface very close to a given plane with
irrational slope in $\mathbb{Z}^5$~\cite{dB-81}.

Let $(e_1,e_2,\ldots, e_d)$ be the canonical basis of $\mathbb{Z}^d$ and
$\overline{\alpha}_1, \overline{\alpha}_2,\ldots,\overline{\alpha}_d$ be the
angles
associated to
the train-tracks such that $e^{i\overline{\alpha}_j}$ is the projection of $e_j$, and
\begin{equation}
\label{eq:ordering_angles}
\overline{\alpha}_1 < \overline{\alpha}_2 < \cdots <
\overline{\alpha}_d < \overline{\alpha}_1+\pi < \cdots < \overline{\alpha}_d+\pi <
  \overline{\alpha}_1+2\pi.
\end{equation}

Using this notion of surface and the notation above, we can give a sense to the
difference of two vertices $y-x$ as the vector $\sum_{j=1}^d N_j e_j$ in $\ZZ^d$
joining the corresponding points on the surface. The quantity
\begin{equation*}
N=\vert y-x\vert=\sum_{j=1}^d \vert N_j\vert
\end{equation*} is the graph distance between
$x$ and $y$ seen as vertices of $\GR$, as well as the graph distance in $\ZZ^d$ between the corresponding points on the monotone surface.
The \emph{reduced coordinates} of $y-x$
are the quantities $n_j=\frac{N_j}{N}$, which define a point
$\pps=\sum_{j=1}^{d} n_j e_j$ in the $L^1$-unit ball of
$\mathbb{R}^d$.

Monotone surfaces can be wild in the following sense: fix a reference
vertex $x_0$ and a direction $\ps\in\mathbb S^1$, and look at the
reduced coordinates $n_j$ in a minimal path from
$x_0$ to $y$, as $y$ tends to infinity in the direction $\ps\in\mathbb S^1$ in the
plane embedding. These quantities may not converge. This is what happens for
example on Figure~\ref{fig:waves}, because of the larger and larger ``waves''
causing oscillations for the $n_j$'s. 

We say that the isoradial graph
is \emph{asymptotically flat} if this does not happen, and if proportions
converge for each directional limit. This property,
and the values of the limits when they exist,
do not depend on the choice
of $x_0$. Periodic isoradial graphs and multiple cone-like regions (see
Figure~\ref{fig:Iso0} and Section~\ref{sec:cones}) are asymptotically flat.
Let us call $n_j(\ps)$
the limits
for $j=1,\ldots,d$
of the reduced coordinates in the direction $\ps$.
In that case, the function
\begin{equation*}
  \ps\mapsto \vec{n}(\ps)= (n_j(\ps))_{1\leq j\leq d}
\end{equation*}
is automatically continuous and injective, as the map from the diamond graph
to the monotone surface in $\mathbb{Z}^d$ is a bi-Lipschitz bijection.

\subsection{Exponential functions}
\label{sec:expo_def}
We will be using definitions and results on the $Z$-invariant massive Laplacian $\Delta:\CC^\Vs\rightarrow\CC^\Vs$ introduced in~\cite{BodTRa-17}\footnote{The presentation of Section \ref{sec:expo_def} follows that of \cite[Sec.~3]{BodTRa-17}.}. 
Let $\xb$ be a vertex of $\Gs$ of degree $n$; denote by $e_1,\dots,e_n$ the edges incident to $\xb$ and by
$\overline{\theta}_1,\dots,\overline{\theta}_n$ the corresponding rhombus half-angles, then the Laplacian operator is given by \eqref{eq:Laplacian_operator} (see \cite[Eq.~(1)]{BodTRa-17} for the original statement),
where the conductances $\rho$ and (squared) masses $m^{2}$ are defined by
\begin{align}\label{eq:conductances}
     \rho_e=\rho(\overline{\theta}_e\vert k)&=\sc(\theta_e\vert k),\\
     m^2(\xb\vert k)&=\sum_{j=1}^{n}(\Arm (\theta_j\vert k)-\sc(\theta_{j}\vert k)),\label{eq:mass}
\end{align}
where $\sc$ is one of the twelve Jacobi elliptic functions (see \cite[Chap.~16]{AbSt-64} or \cite[Chap.~2]{La-89} for 
extensive presentations,
or \cite[App.~A]{BodTRa-17} for a selection of key properties) and
\begin{equation*}
     \Arm (u\vert k)=
     \frac{1}{k'}\Bigl(
     \Dc(u\vert k)+\frac{E-K}{K}u
     \Bigr),
\end{equation*}
where $k'=\sqrt{1-k^2}$ is the complementary elliptic modulus, $\Dc(u\vert
k)=\int_{0}^u \dc^2(v\vert k)\ud v$, and $E=E(k)$ is the complete elliptic
integral of the second kind. Since $k$ is fixed once and for all, most of the time we will drop the elliptic modulus from our notation, writing for example $\Arm (u)$ instead of $\Arm (u\vert k)$.

We also need the definition of the complex-valued
\emph{discrete $k$-massive exponential function} $\expo_{(x,y)}(u)$
of \cite[Sec.~3.3]{BodTRa-17}, briefly introduced in \eqref{eq:intro_expo},
depending on a pair of vertices $(x,y)$ and of a complex parameter $u$.
Consider an edge-path
$\xb=\xb_1,\dotsc,\xb_n=\yb$ of the diamond graph $\GR$ from $\xb$ to $\yb$
and let
$e^{i\overline{\alpha}_j}$ be the vector corresponding to the edge
$\xb_j\xb_{j+1}$. Then the exponential function is defined inductively 
along the edges of the path:
\begin{align}
\label{eq:recursive_def_expo}
\forall\,u\in\CC,\quad
\expo_{(\xb_j,\xb_{j+1})}(u) &= i \sqrt{k'}\sc\Bigl(\frac{u-{\alpha_j}}{2}\Bigr),\nonumber\\
\expo_{(\xb,\yb)}(u)         &= \prod_{j=1}^{n-1} \expo_{(\xb_j,\xb_{j+1})}(u),
\end{align}
where $\alpha_j=\overline{\alpha}_j\frac{2K}{\pi}$.
These functions do not depend on the path chosen for their definition,
and are harmonic for the
Laplacian~\eqref{eq:Laplacian_operator}, see~\cite[Prop.~11]{BodTRa-17}.

\subsection{Exact and asymptotic expressions for the Green function}
\label{subsec:ex_as_ex}

The massive Green function, denoted $G$, is the inverse of the massive Laplacian
operator~\eqref{eq:Laplacian_operator}. The following formula is proved in
\cite[Thm~12]{BodTRa-17}:
\begin{equation}
\label{eq:def_Green_function}
     G(\xb,\yb) =\frac{k'}{4i\pi} \int_{\Gamma_{\xb,\yb}} \expo_{(\xb,\yb)}(u) \ud u,
\end{equation}
where $\Gamma_{\xb,\yb}$ is a vertical contour on the torus 
$\TT(k)$, whose abscissa can be identified with the angle of the ray $\RR\overrightarrow{\xb\yb}$.
Let us remark that the  discrete massive exponential function $\expo_{(\xb,\yb)}(u)$ in \eqref{eq:recursive_def_expo} and \eqref{eq:def_Green_function} is defined using a path of the embedded graph from $\xb$ to $\yb$. This implies that the expression \eqref{eq:def_Green_function} for $G(\xb,\yb)$ is local, meaning that it remains unchanged if the isoradial graph $\Gs$ is modified away from a path from $\xb$ to $\yb$. This is far from being a general situation: when computing the inverse of a discrete operator, one expects the geometry of the whole graph to be involved. The idea of the proof of the local formula \eqref{eq:def_Green_function} is the following:  find a one-parameter family of local, complex-valued functions in the kernel of the massive Laplacian \eqref{eq:Laplacian_operator} (the exponential functions), define its inverse (minus the Green function) as a contour integral of these functions, and adjust (if possible) the contour of integration in such a way that $\Delta G=-\Id$.

Let us also state \cite[Thm 14]{BodTRa-17}, which contains
the asymptotic behavior of the Green function. Let $\Gs$ be a
quasicrystalline isoradial graph. Let $\chi$ the function defined (for $\xb$ and
$\yb$ fixed) by
\begin{equation}
\label{eq:def_chi}
     \chi(v)=\chi(v\vert k) = \frac{1}{\vert \xb-\yb\vert}\log \expo_{(\xb,\yb)}(v+2iK'),
\end{equation}
where $\vert \xb-\yb\vert$ is the graph distance between $x$ and $y$.

By \cite[Lem.~17]{BodT-11} the set
of zeros of
$\expo_{(\xb,\yb)}(u)$ is contained in an interval of length $2K-2\epsilon$, for
some $\epsilon>0$. Let us denote by $\alpha$ the midpoint of this interval. When
the distance $\vert \xb_0-\yb\vert$ between vertices $\xb_0$ and $\yb$ of $\Gs$ is large, we have
\begin{equation}
\label{eq:asymptotics_BodTRa-17}
     G(\xb_0,\yb)=\frac{k' \expo_{(x_0,y)}(2iK'+v_0)}{2\sqrt{2\pi \vert \xb_0-\yb\vert \chi''(v_0)}} (1+o(1)),
\end{equation}
where $v_0$ is the unique $v\in\alpha +(-K+\varepsilon,K-\varepsilon)$ such that $\chi'(v)=0$, see \cite[Lem.~15]{BodTRa-17}, and $\chi(v_0)<0$. 
Let us briefly recall the main idea in \cite{BodTRa-17} to derive the asymptotics \eqref{eq:asymptotics_BodTRa-17}. Starting from \eqref{eq:def_chi}, we may reformulate \eqref{eq:def_Green_function} as
\begin{equation*}
     G(\xb_0,\yb) =\frac{k'}{4i\pi} \int_{\Gamma_{\xb_0,\yb}} e^{\vert \xb_0-\yb\vert \chi(v)} \ud v.
\end{equation*}
Typically, as $\vert \xb_0-\yb\vert\to\infty$, one may analyse the above integral using the saddle-point method. Classically, the saddle point is a critical point, therefore $\chi'(v)=0$.

\subsection{Going to infinity in an isoradial graph and 3D-consistency} 
\label{subsec:going_to_infinity}

Let us first recall that by construction \cite{Do-59,Hu-60,Sa-97,Wo-00}, a sequence of points $y$ 
converges to a point in the Martin boundary if it exits any finite subset of the graph and if the ratio of Green functions
\begin{equation*}
     \frac{G(x_1,y)}{G(x_0,y)}
\end{equation*}
converges pointwise. While \eqref{eq:asymptotics_BodTRa-17} provides the main
term in the Green function asymptotics and is independent on how $y$ goes to
infinity in the graph, we thus need here a different, somehow more precise
information. 

In the classical Ney and Spitzer theorem \eqref{eq:asymptotic_Martin_kernel_NS},
$y$ goes to infinity along an angular direction, namely $\frac{y}{\vert
y\vert}\to \ps\in\mathbb S^1$. However, this simple geometric description will
not work in the isoradial setting. Indeed, it can be shown (see the proof of
\cite[Thm~14]{BodTRa-17}) that the convergence of $v_0$ appearing in
\eqref{eq:asymptotics_BodTRa-17} is equivalent to the convergence of the
reduced coordinates, which as already discussed (see Section \ref{subsec:monotone} and Figure~\ref{fig:waves}) is not guaranteed by the angular convergence of $y$.

In our opinion, the best way to characterize the convergence of the Martin 
kernel is to use the concept of \emph{3D-consistency}~\cite{BoSu-08}, that we now introduce.
The notion of harmonicity for our massive Laplacian \eqref{eq:Laplacian_operator} 
is compatible with the star-triangle transformation (see
Figure~\ref{fig:triangle_etoile}), which corresponds in the monotone surface
picture to pushing the surface along a 3-dimensional cube:
let $\Gs_Y$ and $\Gs_\nabla$ be two isoradial graphs differing by a star-triangle
transformation, such that $\Gs_Y$ has an extra vertex $x_0$ of degree 3.
If $f$ is a harmonic function on $\Gs_\nabla$, there is a unique way to
extend it at $x_0$ to make it harmonic on $\Gs_Y$. This operation,
together with its inverse corresponding to forgetting the value at $x_0$,
realize
a bijection between harmonic functions 
on $\Gs_Y$ and $\Gs_\nabla$; see~\cite[Prop.~8]{BodTRa-17}.

\begin{figure}
  \centering
  \includegraphics[angle=180,width=6cm]{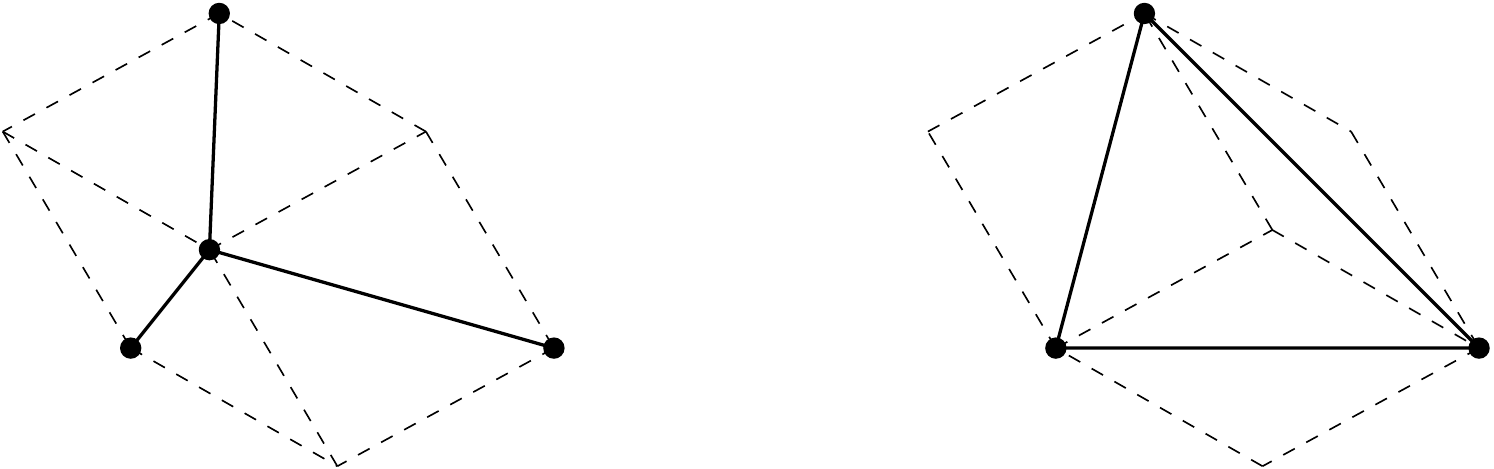}
  \caption{A star-triangle transformation on isoradial graphs. $\Gs_\nabla$ is on the
    left, $G_Y$ on the right.}
    \label{fig:triangle_etoile}
\end{figure}

Using the 3D-consistency and the interpretation of quasicrystalline isoradial graphs as monotone surfaces 
on $\mathbb Z^d$ (Section \ref{subsec:monotone}), the exponential function and the Green function can
be uniquely lifted to
vertices of $\mathbb{Z}^d$, keeping their harmonicity properties.
More precisely, if we assume that the isoradial graph $\mathsf{G}$ has enough
train-tracks associated to each direction of $\mathbb{Z}^d$, namely that
each coordinate on the monotone surface is not bounded either from above or
below\footnote{%
  With
  the $n_j$'s denoting the reduced coordinates of $y-x$, where $x$ is some reference
  vertex, 
  this is equivalent to asking that for each $j$ and $\pm$, there is at least a
  direction $\ps$  along which $\limsup \pm n_j$ is
  strictly positive as $y$ tends to infinity along $\ps$, or in the
  planar embedding picture, asking that for any train-track $T$ associated to
  a direction $\alpha$, the two half-planes obtained by cutting along the
  parallel edges of $T$ have both an infinite number of train-tracks with this
  direction $\alpha$.
}, one can generate any monotone surface of $\mathbb{Z}^d$ from $\mathsf{G}$ by
a (maybe infinite) number of star-triangle transformations, and can extend the
exponential and Green functions to all even vertices of $\mathbb{Z}^d$.
Otherwise, one may get only
half-spaces or slabs~\cite{BoFeRe}, but which are bi-infinite in $d'\geq 2$ directions.
 We will often assume that we are in the first situation, in which $\Gs$ spans all
 $\mathbb{Z}^d$. To extend our results to the second case, replacing $d$ by $d'$
 in the proofs will often be enough.

This lift to a lattice of higher dimension allows us to reformulate the
asymptotics of the Green function as a limit as $y$ tends to infinity in a
certain direction, thereby showing the strong parallel with the classical statement of Ney and Spitzer theorem.
Indeed, in $\mathbb{Z}^d$ the
convergence of $y$ along a direction is equivalent to the convergence of the
reduced coordinates. These directions to infinity may not be obtained by staying
on the surface corresponding to the graph $\Gs$ (unless for some directions, if
the graph is asymptotically flat).

In the rest of the manuscript, we will always use this notion of convergence
when we will compute the asymptotics of ratios of Green functions.

\section{Main results and proofs}
\label{sec:main_result}

\subsection{Statements}

\begin{thm}[asymptotics of the Martin kernel]
\label{thm:main-1}
Let $G$ be the massive Green function on a quasicrystalline isoradial graph,
which can be lifted to $\mathbb{Z}^d$ via the star-triangle transformation.
Then 
\begin{equation}
\label{eq:asymptotic_Martin_kernel}
     \lim_{\yb \to\infty}\frac{G(\xb_1,\yb)}{G(\xb_0,\yb)}=\expo_{(\xb_0,\xb_1)}(u_0),
\end{equation}
where $u_0=2iK'+v_0$, with $v_0$ defined as in~\eqref{eq:asymptotics_BodTRa-17}.
\end{thm}

\begin{thm}[structure of the Martin boundary]
\label{thm:main-2}
As the asymptotic direction of $\yb \to\infty$ in \eqref{eq:asymptotic_Martin_kernel} varies in $\mathbb S^{d-1}$, 
the point $u_0$ describes entirely the
circle $2iK'+\mathbb{R}/4K\mathbb{Z}$
in $\TT(k)$.
If moreover $\Gs$ is asymptotically flat, the Martin boundary 
 is homeomorphic to the circle $\mathbb S^1$.
\end{thm}

\begin{thm}[minimality of the Martin boundary]
\label{thm:main-3}
Let $x_0\in \Gs$ be any fixed point in the graph and $v\in\mathbb
R/4K\mathbb{Z}$.
The positive harmonic functions $h_v(x)=\expo_{(\xb_0,x)}(2iK'+v)$ are minimal.
\end{thm}

\subsection{Proofs}

\begin{proof}[Proof of Theorem \ref{thm:main-1}]
The starting point is the asymptotics \eqref{eq:asymptotics_BodTRa-17} obtained in \cite{BodTRa-17}, that we first apply to the Green function $G(\xb_0,\yb)$. To derive the asymptotics of $G(\xb_1,\yb)$ we begin by rewriting  \eqref{eq:def_Green_function} as
\begin{equation*}
     G(\xb_1,\yb) =\frac{k'}{4i\pi} \int_{\Gamma_{\xb_1,\yb}} \expo_{(\xb_1,\xb_0)}(u)\expo_{(\xb_0,\yb)}(u) \ud u.
\end{equation*}
We then perform an asymptotic analysis of the above integral. The only difference with the proof of \cite[Thm~14]{BodTRa-17} is the presence of the term $\expo_{(\xb_1,\xb_0)}(u)$, compare with \eqref{eq:def_Green_function}. Using the saddle point method and similar ideas as in \cite{BodTRa-17}, we obtain that this factor will actually appear as a prefactor in the final asymptotics, evaluated at $u_0$. In other words,
\begin{equation*}
     G(\xb_1,\yb)=\expo_{(\xb_1,\xb_0)}(u_0)\frac{k' \expo_{(x_0,y)}(u_0)}{2\sqrt{2\pi \vert \xb_0-\yb\vert \chi''(v_0)}}(1+o(1)),
\end{equation*}
with the exact same $u_0=2iK'+v_0$ as in \eqref{eq:asymptotics_BodTRa-17}. We then easily reach the conclusion that the ratio of Green function behaves as in \eqref{eq:asymptotic_Martin_kernel}, which completes the proof.
\end{proof}

\begin{proof}[Proof of Theorem~\ref{thm:main-2}]
    There is correspondence between the direction along which $y$ goes to
    infinity and $u_0=2iK'+v_0$, see \cite[Lem.~15]{BodTRa-17} and its proof.
    This $v_0$ can be anywhere on the circle $\mathbb{R}/4K\mathbb{Z}$: the interval
    from the asymptotics of the Green function moves with the set of poles of
    the exponential function, which depends on the direction of $y$.

    Now we assume
    $\Gs$ to be asymptotically flat, so the application $\ps\mapsto
    \vec{n}(\ps)=(n_j(\ps))_{1\leq j \leq d}$ is well defined and continuous.
    It is even
    differentiable everywhere, except for $\ps$ for which at least one of the
    $n_j(\ps)$ is zero (meaning that the corresponding minimal path seen on the
    monotone surface is moving to another orthant of $\mathbb{Z}^d$), but even
    at these points it has left and right derivatives.

    We want to prove that $\ps\mapsto u_0(\ps)$ is differentiable (has one-sided
    derivatives) when $\ps\mapsto n_j(\ps)$ is differentiable (has one-sided
    derivatives), and that the derivative is strictly positive, in the sense
    that the points move in the same direction around their respective circles.

    Up to a possible cyclic relabelling of the coordinates, we can assume that
    all the $n_j(\ps)$ are non-negative, and we look at a variation of $\ps$ in
    the positive direction.
    Looking at how frequencies of steps along an infinite path should vary as
    the asymptotic direction in the plane moves slightly in the positive
    direction while the path stays in the same orthant, we see that the (one-sided)
    derivative $\frac{d\vec{n}}{d \ps}$ of the reduced coordinates is
    a linear combination with positive coefficients of vectors of the form
    \begin{equation*}
     \delta n_{k,l}= (0,\ldots,0,-1,0,\ldots,0,1,0),
    \end{equation*}
    where the $1$ is at position $l$ and the $-1$ at position $k<l$, namely,
    \begin{equation*}
      \frac{d\vec{n}}{d\ps} = \sum_{k<l} a_{k,l}(\ps)\delta n_{k,l},
    \end{equation*}
    with $a_{k,l}(\ps)\geq 0$, and not all equal to zero.

    Rewrite the condition that $v_0=u_0-2iK'$ is a critical point as in the
    proof of~\cite[Lem.~15]{BodTRa-17}:
    \begin{equation*}
      \chi'(v_0)=\sum_{j=1}^d n_j
      \frac{ \sn \cdot  \cn }{ \dn } \left(\frac{v_0-\alpha_j}{2}\right) =
      \vec{n}\cdot \vec{F}(v_0)=0,
    \end{equation*}
    with $ \vec{F}(v)=  (\frac{\sn \cdot
    \cn}{\dn}(\frac{v-\alpha_j}{2}))_{1\leq j\leq d}$.
    The implicit function theorem implies that $v_0$, as a function of $\ps$, has a
    (one-sided) derivative $\frac{dv_0}{d\ps}$, and it satisfies:
    \begin{equation*}
      \frac{d v_0}{d\ps} \chi''(v_0) + \frac{d \vec{n}}{d \ps} \cdot
      \vec{F}(v_0)=0.
    \end{equation*}
    Now, writing down explicitly the scalar product
    \begin{equation*}
      \frac{d\vec{n}}{d\ps}\cdot \vec{F}(v_0)=
      \sum_{k<l} a_{k,l}
      \left(%
	\frac{\sn\cdot\cn}{\dn}\Bigl(\frac{v_0-\alpha_l}{2}\Bigr)
	-
	\frac{\sn\cdot\cn}{\dn}\Bigl(\frac{v_0-\alpha_k}{2}\Bigr)
      \right),
    \end{equation*}
    we see that this quantity is strictly negative, as all the quantities
    $\frac{v_0-\alpha_j}{2}$ are in the interval~$[-K,K]$, on which the function
    $\frac{\sn\cdot\cn}{\dn}$ is strictly increasing (recall that
    $\alpha_k<\alpha_l$).

    Since $v_0$ corresponds to a simple critical point of $\chi$, which is
    actually a local minimum, $\chi''(v_0)$ is strictly positive. As a
    consequence,
    \begin{equation*}
      \frac{d v_0}{d\ps} >0.
    \end{equation*}


    This implies that the application
    \begin{equation*}
      T:\ps\in\mathbb{S}_1 \mapsto u_0 \in 2iK'+\mathbb{R}/4K\mathbb{Z}
    \end{equation*}
    is continuous, and lifts to a strictly increasing function on the universal
    cover.
    Moreover, we know \cite[Lem.~17]{BodT-11} that the interval of length $2K-2\varepsilon$ in which
    the real part of $u_0$ lies is winding once around the circle.
    Indeed, it should contain the zeros of the exponential function from the
    base vertex $x$ to $y$. When $y$ switches to the next orthant, while
    going around $x$, the sliding interval is dropping the $\alpha_j$ closer
    to its right boundary, and swallows $\alpha_j+\pi$. Therefore, after a full
    turn around $x$, the left and right boundaries of this sliding interval made
    also exactly one full turn.
    Since the total length between the left end of this interval at the
    beginning of the full turn and its right end at the end is strictly less
    than twice the length of the circle, it means that $T$ winds also exactly
    once around the circle, therefore $T$
    is bijective. Since the image is compact, $T$ is automatically open, so it
    defines a homeomorphism.
\end{proof}

\begin{proof}[Proof of Theorem~\ref{thm:main-3}]
It will be given in Section \ref{sec:minimal_harmonic_functions}.
\end{proof}

\subsection{An example: cones of homogeneities}
\label{sec:cones}

The simplest inhomogeneous extension of the fully homogeneous case of $\mathbb Z^d$ consists in modifying the jumps at a finite number of sites, as explained in our introduction (see also \cite{KuMa-98}). Another simple extension is to split the grid in a finite union of connected domains, for example two half-planes, or more generally a finite number of cone-like regions as on Figure \ref{fig:Iso0}, and to assign to each region a (different) set of transition probabilities. Up to our knowledge, a Ney and Spitzer theorem is not known in this simple setting, even in the simplest instance of two half-planes. It is worth mentioning that our approach allows to study such cases, provided the transition probabilities follow the construction of elliptic conductances as introduced in Section \ref{sec:iso_expo}.

\section{The case of periodic isoradial graphs}
\label{sec:periodic_case}

In Section~\ref{sec:amoeba}, we first
study general periodic isoradial graphs and relate important quantities needed in their analysis (characteristic polynomial, amoeba, degeneracy of the amoeba, quantity $u_0$ in \eqref{eq:asymptotic_Martin_kernel_intro}, etc.)\ to natural objects in $\mathbb Z^d$ (generating function of the transition probabilities, set $\partial D_t$ in \eqref{eq:def_partial_D_t}, spectral radius, mapping $u(p)$ in \eqref{eq:asymptotic_Martin_kernel_NS}, etc.).
Then we show in Section~\ref{sec:examples} how our results give another proof of the classical Ney and Spitzer in two cases (the square and triangular lattices).

\subsection{Amoeba of the characteristic polynomial and jump generating function} 
\label{sec:amoeba}

Using the theory of analytic combinatorics in several variables \cite{PeWi}, we
analyze the exponential decay of
the Green function through Fourier analysis. We denote
a vertex of $\Gs$ by a triple $(v,n_1,n_2)$, where $v$ is the copy of the
vertex in the fundamental domain and $(n_1,n_2)$ is the element of $\mathbb{Z}^2$
corresponding to the translation moving $v$ to the vertex.

\subsubsection*{A Ney-Spitzer theorem for periodic planar graphs}
The first part of this section is not specific to isoradial graphs and holds
for all planar periodic graphs (having periodic masses and conductances) with
at least one positive mass.

The Laplacian can be seen as a periodic convolution operator on vector-valued
functions of $\mathbb{Z}^2$. It acts in Fourier space by multiplication by a
matrix $\Delta(z,w)$, with $\vert z\vert=\vert w\vert=1$.
The square matrix $\Delta(z,w)$ has rows
and columns indexed by vertices of the fundamental domain, and the coefficient
between $v$ and $w$ is the sum of the conductances of the edges between
$(v,0,0)$ and $(w,n_1,n_2)$ multiplied by $z^{n_1} w^{n_2}$.

If $(v,n_1,n_2)$ and $(w,n'_1,n'_2)$ are two vertices of $\Gs$, the Green function can
be expressed as the Fourier inverse transform of
$\Delta(z,w)^{-1}=\frac{Q(z,w)}{P(z,w)}$,
where $P(z,w)=\det \Delta(z,w)$ and $Q(z,w)$ is the adjugate matrix of
$\Delta(z,w)$:
\begin{equation*}
     G((v,n_1,n_2),(w,n'_1,n'_2))= \iint _{\vert z\vert=\vert w\vert=1}
     \frac{z^{n_1-n'_1} w^{n_2-n'_2}Q(z,w)_{v,w}}{P(z,w)}
     \frac{\ud z}{2i\pi z}\frac{\ud w}{2i\pi w}.
\end{equation*}
The asymptotic behavior of the Green function above is encoded in the singularities
of the integrand, namely, the zeros of $P$. The zero-set of $P$ defines an
algebraic curve, called the \emph{spectral curve} of the Laplacian.
The \emph{amoeba} of $P$ (see \cite[Chap.~6]{GeKaZe}) is
the image of the spectral curve by the application 
\begin{equation*}
     (z,w)\mapsto(\log\vert z\vert,\log\vert w\vert)
\end{equation*}     
and plays an important
role in the discussion; see Figure~\ref{fig:amoeba} for an example.

\begin{figure}[ht]
\includegraphics[width=0.4\textwidth]{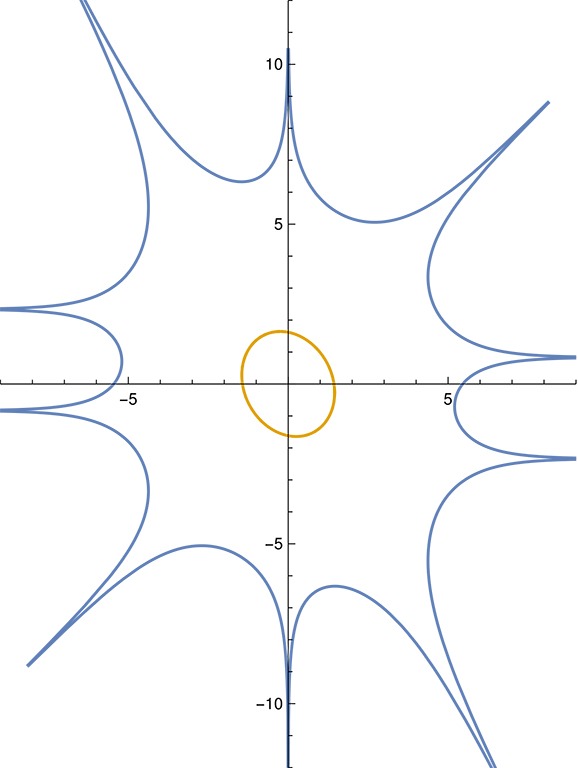}
\caption{Amoeba of the spectral curve of a Laplacian with mass on a periodic
  graph. The amoeba itself is the region between the yellow and blue curves. The
  curve here has geometric genus one, and the
  (unique in this case) bounded connected component of the
complementary of the amoeba
containing the origin is called the \emph{oval}.}
\label{fig:amoeba}
\end{figure}

In the massive case, there is no zero of $P$ on the unit torus $\vert z\vert=\vert w\vert=1$, and
thus we can deform continuously the contour of integration defining $G$ to
increase or decrease the radii of the torus, until we touch the spectral curve,
in order to obtain exponential decay.

We can say even more, because of
geometric properties of spectral curves of Laplacians.
The spectral curve is always a simple Harnack
curve~\cite{Ke-19}. In particular, the boundary of the amoeba is the image of
the real locus of the spectral curve. The bounded connected component of the
complementary of the amoeba
containing the origin (called the oval $\mathcal{O}$ in the sequel, 
see Figure \ref{fig:amoeba}) is a
convex set whose boundary corresponds to a (subset) of positive real roots of~$P$:
\begin{equation*}
     P(e^{\zeta_1},e^{\zeta_2})=0,
\end{equation*}
which in the classical theorem by Ney and Spitzer 
is the generalization of the level set $\partial D_t$ 
defined in \eqref{eq:def_partial_D_t}.
All conditions are met to be in the smooth point situation described
in~\cite[Chap.~9]{PeWi}, allowing us to readily obtain the asymptotics of the
ratio of Green functions for copies of the same vertex $v$ in the fundamental
domain.

Let $x_0=(v,0,0)$, $x_1=(v,n_1,n_2)$ and $y=(v,n'_1,n'_2)$.
As $(n'_1,n'_2)$ goes to infinity in $\mathbb{Z}^2$ in the direction
$\ps$, we get the following asymptotics for the Green function:
\begin{equation*}
  \lim_{y\to\infty}\frac{G(x_1,y)}{G(x_0,y)} = e^{\zeta\cdot n},
\end{equation*}
where $\zeta=\zeta(\ps):=\arg\max \{\ps\cdot s : s \in
\mathcal{O}\}$, 
which by convexity is reached along the boundary of $\mathcal{O}$.
The homeomorphism $\psi\colon\ps\in\mathbb{S}^1\mapsto \zeta(\ps)\in\partial
\mathcal{O}$ is the analogue of the Ney and Spitzer function $\zeta$ in
the fully transitive case, which we recall here for completeness:

\begin{lem}[\cite{He-63}]
\label{lem:He-63}
Let $\rho$ denote the spectral radius of the transition kernel. Assume $t\geq\rho$, with $t\neq \rho$ if the random walk has zero drift. Then the set $D_t=\{\zeta\in\mathbb R^d: \phi(\zeta)\leq t\}$ is compact and convex, the gradient
\begin{equation*}
     \textnormal{grad } \phi(\zeta)=\sum_{x\in\mathbb Z^d} x e^{\zeta\cdot x} P(0,x)
\end{equation*}
exists everywhere on $D_t$ and does not vanish on its boundary $\partial D_t$. Furthermore, the mapping
\begin{equation}
\label{eq:grad_p}
     \zeta\mapsto \ps=\frac{\textnormal{grad } \phi(\zeta)}{\vert\textnormal{grad } \phi(\zeta)\vert}
\end{equation}
determines a homeomorphism between $\partial D_t$ and $\mathbb S^{d-1}$.
\end{lem}

The following result is stated under Corollary 1.3 in \cite{NeSp-66} for $t=1$ and random walks with drift; it is extended to other values of $t$ and to the zero drift case in \cite[(25.21)]{Wo-00}.

\begin{thm}
\label{thm:NS}
Under the same assumptions as in Lemma \ref{lem:He-63}, let $\ps$ be a point in $\mathbb S^{d-1}$ and $\zeta(\ps)$ be the unique solution in $\partial D_t$ of \eqref{eq:grad_p}.
Let $y$ denote any sequence in $\mathbb Z^d$ such that $\vert y\vert \to\infty$ and $\frac{y}{\vert y\vert}\to \ps$. Then the ratio $G(x_1,y\vert \tfrac{1}{t})/G(x_0,y\vert \tfrac{1}{t})$ converges to $e^{\zeta(\ps)\cdot (x_1-x_0)}$.
\end{thm}

\subsubsection*{The isoradial periodic case}

When in addition the periodic graph $\Gs$ is isoradial, with the elliptic
conductances and masses from \eqref{eq:conductances} and \eqref{eq:mass}, the
spectral curve has geometric genus $1$ and there is a unique bounded connected component
of the complementary of the amoeba~\cite{BodTRa-17}. Furthermore, we have an explicit
parameterization of its boundary via the exponential function:
\begin{equation}
\label{eq:para}
     \xi\colon u\mapsto \bigl(\log\vert e_{x,x+(1,0)}(u)\vert,\log\vert e_{x,x+(0,1)}(u)\vert\bigr),
\end{equation}
where $u\in 2iK'+ \mathbb{R}/4K\mathbb{Z}$. The composition
$\xi^{-1}\circ\psi$ is the application giving the $u_0$ from the asymptotics
of Green function of Equation~\eqref{eq:asymptotics_BodTRa-17} for every
direction $\ps$ along which $y$ can go to infinity.

\subsection{Examples}
\label{sec:examples}

Among all periodic isoradial graphs, two of them (namely, the square and triangular lattices) are of special interest, as they correspond to models of homogeneous random walks in $\mathbb Z^2$. In this section we show how, in these two cases, our results match with the classical Ney and Spitzer theorem. 

In order to state precisely this connection, it is important to notice that in Ney and Spitzer framework, a set of transition probabilities is fixed (through the function $\phi$ in \eqref{eq:def_phi}) and then the $t$-Martin boundary is computed for any fixed $t$ larger than the spectral radius. On the other hand, in the isoradial setting the conductances (related to the transition probabilities) and the mass (related to the variable $t$) both depend on the same elliptic modulus $k$, and it is a priori unclear that we can construct a random walk model with transition probabilities independent of the $t$ variable. This will be shown in the two above-mentioned examples, using the degrees of freedom that we have on the conductances and on the elliptic modulus.

\subsubsection*{The square lattice}
The first example is the square lattice with angles $\alpha$ and
$\beta$ associated to the two families of parallel train-tracks, with $\alpha <
\beta < \alpha+2K < \beta+2K$, see
Figure~\ref{fig:Iso0}. Because the definition of conductances and
masses is invariant by rotation of the whole embedding, one can further assume
that $\beta=-\alpha=\theta$ (see Figure \ref{fig:rhombus_angle}).
There are thus two types of edges: horizontal
with opening angle $\overline{\theta}=\frac{\pi}{2K}\theta\in(0,\frac{\pi}{2})$ and vertical with
opening angle
$\frac{\pi}{2}-\overline{\theta}$.
The characteristic polynomial of the model is given by \cite[Sec.~5.2]{BodTRa-17}
\begin{equation}
\label{eq:charac_poly_square}
     P(z,w)=m+2(c_1+c_2)-c_1\left(z+\frac{1}{z}\right)-c_2\left(w+\frac{1}{w}\right),
\end{equation}
with $c_1=\sc(\theta)$, $c_2=\sc(K-\theta)$ and $m$ as in \eqref{eq:mass}. It is convenient to normalize \eqref{eq:charac_poly_square} so as to have transition probabilities summing to one:
\begin{equation*}
       \widetilde P(z,w)=\frac{m+2(c_1+c_2)}{2(c_1+c_2)}-p_1\left(z+\frac{1}{z}\right)-p_2\left(w+\frac{1}{w}\right),
\end{equation*}
with 
\begin{equation}
\label{eq:def_p1_p2}
     p_1=\frac{c_1}{2(c_1+c_2)} \quad \text{and} \quad p_2=\frac{c_2}{2(c_1+c_2)}=\frac{1}{2}-p_1.
\end{equation}
It is natural to associate to this model the classical simple random walk on $\mathbb Z^2$, with transition probabilities Laplace transform
\begin{equation*}
     \phi(\zeta_1,\zeta_2)=p_1(e^{\zeta_1}+e^{-\zeta_1})+p_2(e^{\zeta_2}+e^{-\zeta_2}).
\end{equation*}
Notice that horizontal (resp.\ vertical) jumps have the same weight, due to the constraints on the conductances in \eqref{eq:charac_poly_square} (in particular, the random walk has zero drift).
We now state precisely the connection between the two frameworks.

\subsubsection*{Transition probabilities generating function, characteristic polynomial and their zero-sets}

It is clear that $\widetilde P(z,w)=0$ if and only if $\phi(\log z,\log w)=t$, with $t=1+m/(2c_1+2 c_2)$. 
In particular, using the notation of Section \ref{sec:amoeba}, the oval of the amoeba is in correspondence with the set $\partial D_t$.

\subsubsection*{Exponential functions}
As said in the introduction, extremal positive harmonic functions in group
structures are exponential functions \cite{ChDe-60}. In our context of the
planar lattice, they are given by (writing $\zeta=(\zeta_1,\zeta_2)$ and $n=(n_1,n_2)$)
\begin{equation}
\label{eq:formula_expo_Z2_NS}
     f(x)=e^{\zeta\cdot n}=\left\{e^{\zeta_1}\right\}^{n_1}\left\{e^{\zeta_2}\right\}^{n_2},
\end{equation}     
for $\zeta\in\partial D_t$. By definition of $\partial D_t$, see \eqref{eq:def_partial_D_t}, the pairs $(e^{\zeta_1},e^{\zeta_2})$ parametrize the zero-set of $\phi-t$.

To compare the elementary expression \eqref{eq:formula_expo_Z2_NS} for the exponential function to its elliptic analogue \eqref{eq:recursive_def_expo}, we should first introduce an elliptic uniformisation of the zero-set of the characteristic polynomial \eqref{eq:charac_poly_square}. As shown in \cite[Eq.~(32)]{BodTRa-17} (see also \eqref{eq:para}),
\begin{equation*}
     \{(z,w)\in (\mathbb C\cup\{\infty\})^2 : P(z,w)=0\}=\{(z(u),w(u)): u\in\TT(k)\},
\end{equation*}
with 
\begin{equation*}
     z(u)=-k' \sc\Bigl(\frac{u-\alpha}{2}\Bigr)\sc\Bigl(\frac{u-\beta}{2}\Bigr) 
     \quad \text{and} \quad w(u)= \frac{\sc(\frac{u-\beta}{2})}{\sc(\frac{u-\alpha}{2})}.
\end{equation*}
Moreover, using \eqref{eq:charac_poly_square} one obtains
\begin{equation}
\label{eq:formula_expo_Z2_iso}
     \expo_{(n_1,n_2)}(u) = \left\{i \sqrt{k'}\sc\Bigl(\frac{u-{\alpha}}{2}\Bigr)\right\}^{n_1-n_2} \left\{i \sqrt{k'}\sc\Bigl(\frac{u-{\beta}}{2}\Bigr)\right\}^{n_1+n_2}= z(u)^{n_1}w(u)^{n_2}.
\end{equation}
The comparison between \eqref{eq:formula_expo_Z2_NS} and \eqref{eq:formula_expo_Z2_iso} is now clear: in both cases they are products of the coordinates (raised to some power) of parameterizations of the characteristic polynomial.



\subsubsection*{Choice of the parameters}
Our main result is stated below; it shows that given arbitrary transition probabilities and 
any value of $t$ in the spectral interval, we can adjust the values of $\theta$ and $k$ so as to have the equivalence: $\widetilde P(z,w)=0$ if and only if $\phi(\log z,\log w)=t$. 
\begin{prop}
\label{prop:all_values_square_lattice}
Let $p_1$ and $p_2$ be defined by \eqref{eq:def_p1_p2}. For any $q_1,q_2,t$ such that $q_1,q_2>0$, $q_1+q_2=\frac{1}{2}$ and $t\geq1$, there exist $\theta\in(0,K)$ and $k\in[0,1)$ such that $p_1=q_1$, $p_2=q_2$ and $\frac{m+2(c_1+c_2)}{2(c_1+c_2)}=t$.
\end{prop}

\begin{proof}
We first look at the equation $\frac{c_1}{2(c_1+c_2)}=q_1$. Since $c_2=\sc(K-\theta)=\frac{1}{k'\sc(\theta)}$ by \cite[16.8.9]{AbSt-64}, it is equivalent to
\begin{equation*}
     \frac{\sc(\theta)}{\sc(\theta)+\frac{1}{k'\sc(\theta)}}=2q_1,
\end{equation*}
which results in 
\begin{equation}
\label{eq:value_alpha}
     \sc(\theta)=\sqrt{\frac{2q_1}{1-2q_1}}\frac{1}{\sqrt{k'}},
\end{equation}
which for any $k\in[0,1)$ has a (unique) solution, see \eqref{eq:exp_alpha_square-lattice} for an explicit expression. Now we turn to the equation involving the mass. Note first that by \eqref{eq:mass},
\begin{equation*}
     m=2(\Arm(\theta)+\Arm(K-\theta))-2(c_1+c_2),
\end{equation*}
and by \cite[Eq.~(60)]{BodTRa-17}, $\Arm(\theta)+\Arm(K-\theta)=\frac{\ns(\theta)\dc(\theta)}{k'}$, so that 
\begin{equation*}
     \frac{m+2(c_1+c_2)}{2(c_1+c_2)}=\frac{\nc(\theta)\dc(\theta)}{1+k'\sc(\theta)^2}.
\end{equation*}
Let us now reproduce \cite[16.9.3]{AbSt-64}: ${k'}^2\sc^2+{k'}^2={k'}^2\nc^2=\dc^2-k^2$. These equalities allow us to derive the values of $\nc(\theta)$ and $\dc(\theta)$, starting from the value \eqref{eq:value_alpha} of $\sc(\theta)$. After some computations, we conclude that
\begin{equation*}
     \frac{m+2(c_1+c_2)}{2(c_1+c_2)}=\sqrt{1+2q_1(1-2q_1)\left(k'-2+\frac{1}{k'}\right)}.
\end{equation*}
As $k$ varies in $[0,1)$, the latter function increases from $1$ to $\infty$. In conclusion, $k$ is determined by solving the equation
\begin{equation*}
     \sqrt{1+2q_1(1-2q_1)\left(k'-2+\frac{1}{k'}\right)}=t,
\end{equation*}
see \eqref{eq:exp_t_square-lattice}, and then $\theta$ is found with \eqref{eq:value_alpha}. We can even be more explicit:
\begin{align}
     k&=\frac{t^2-8 q_1^2+4 q_1-1-\sqrt{(t-1) (t+1) (t+1-4 q_1) (t-1+4 q_1)}}{4 q_1 (1-2 q_1)},\label{eq:exp_t_square-lattice}\\
     \theta&=F\left(\sqrt{\frac{2q_1}{k'+2q_1(1-k')}},k\right),\label{eq:exp_alpha_square-lattice}
\end{align}
where $F$ is the incomplete elliptic integral of the first kind. The proof is completed.
\end{proof}

\subsubsection*{The triangular lattice}
The characteristic polynomial is (again by \cite[Sec.~5.2]{BodTRa-17})
\begin{equation*}
     P(z,w)=m+2(c_1+c_2+c_3)-c_1\left(z+\frac{1}{z}\right)-c_2\left(w+\frac{1}{w}\right)-c_3\left(zw+\frac{1}{zw}\right),
\end{equation*}
and our aim is to show an analogue of Proposition \ref{prop:all_values_square_lattice}. However, as the computations become much more complex in this case we will only consider here the case of uniform probabilities, corresponding to three equal conductances $c_1=c_2=c_3=\sc(\frac{K}{3})$. By \eqref{eq:mass}, the normalized characteristic polynomial takes the form
\begin{equation*}
     \widetilde P(z,w)=\frac{\Arm(\frac{K}{3})}{\sc(\frac{K}{3})}-\frac{1}{6}\left(z+\frac{1}{z}\right)-\frac{1}{6}\left(w+\frac{1}{w}\right)-\frac{1}{6}\left(zw+\frac{1}{zw}\right).
\end{equation*}

\begin{prop}
\label{prop:all_values_triangular_lattice}
As $k$ varies in $[0,1)$, $\Arm(\frac{K}{3})/\sc(\frac{K}{3})$ varies continuously in $[1,\infty)$.
\end{prop}

\begin{proof}
Let us first prove that $t=\Arm(\frac{K}{3})/\sc(\frac{K}{3})$ is an algebraic function of $k$, solution to
\begin{equation}
\label{eq:alg_t}
     -27(1-k^2)t^4+18(1-k^2)t^2+2(2-k^2)^2t+1-k^2+k^4=0.
\end{equation}
The first terms of its series expansion are
\begin{equation*}
     t=1+\frac{3}{64}k^4+\frac{3}{64}k^6+\frac{711}{16384}k^8+\frac{327}{8192}k^{10}+O(k^{12}),
\end{equation*}
and further computations suggest that all non-zero coefficients are positive
(see Appendix~\ref{sec:app_AB} for a proof).
Remark that the algebraicity of $t$ is not obvious, as both functions $\Arm$ and $\sc$ are transcendental.

We start by showing that $s=\sn(\frac{K}{3})$ is the unique power series solution to
\begin{equation}
\label{eq:alg_s}
     k^2s^4-2k^2s^3+2s-1=0.
\end{equation}
The first terms in its expansion are
\begin{equation*}
     s=\frac{1}{2}+\frac{3}{32}k^2+\frac{3}{64}k^4+\frac{123}{4096}k^6+\frac{177}{8192}k^8+\frac{34887}{2097152}k^{10}+O(k^{12}).
\end{equation*}
The coefficients of $s$ are again non-negative (see Appendix~\ref{sec:app_AB} for a proof).

To prove \eqref{eq:alg_s} we start from the equality $\sn(K)=1$ that we rewrite as $\sn(\frac{K}{3}+\frac{2K}{3})=1$, and then we apply the classical addition formula \cite[16.17]{AbSt-64} together with some identities in Exercise 22 in \cite[Chap.~2]{La-89}.

The second step is to show that both $\Arm(\frac{K}{3})$ and $\sc(\frac{K}{3})$ may be algebraically expressed in terms of $s$. By definition, $\sc^2=(1/\sn^2-1)^{-1}$, so that $\sc(\frac{K}{3})=(1/s^2-1)^{-1/2}$. Now, using \cite[(60)~and~(61)]{BodTRa-17}, we obtain that 
\begin{equation*}
     \Arm\left(\frac{K}{3}\right)=\frac{\sqrt{1-k^2s^2}}{3s^2\sqrt{1-k^2}}\left(1-\frac{2(1-k^2)s^4}{1-(2-k^2s^2)s^2}\right).
\end{equation*}
Some computations (which we do not reproduce here) finally lead to \eqref{eq:alg_t}.

We prove Proposition \ref{prop:all_values_triangular_lattice}. Computing the
discriminant of the polynomial in \eqref{eq:alg_t}, it is easy to see that for
any $k\in(0,1)$, two roots are real and the other two are non-real, complex conjugate. Moreover, the four roots are obviously non-zero. So the solution $t=t(k)$ is such that $t(0)=1$ (evaluate \eqref{eq:alg_t} at $k=0$). Doing the change of variable $t\to \frac{t}{(1-k^2)^{1/3}}$ in \eqref{eq:alg_t}, we deduce that $t(k)$ behaves as $\frac{1}{(1-k)^{1/3}}$ as $k\to1$, and in particular goes to $\infty$.
\end{proof}

\section{Minimal positive harmonic functions}
\label{sec:minimal_harmonic_functions}

From the construction of Martin boundary, any positive harmonic function $f$ can be
written as an integral of the exponential functions against some Radon measure
$\mu$ supported on $2iK'+\mathbb{R}/4K\mathbb{Z}$:
\begin{equation}
\label{eq:f_expo}
  \forall x\in\Gs,\quad f(x)=\int \expo_{(x,x_0)}(u)\mu(\ud u).
\end{equation}
Let us recall that positive harmonic functions are particularly important from a potential theory point of view, in relation with the concept of Doob transform. The function $f$ is then said to be minimal if we cannot find a non-trivial
measure $\mu$ for this decomposition, i.e., whose support is not just a
singleton; see Section~\ref{subsec:NS}.

In this section, we prove that the
exponential functions $x\mapsto \expo_{(x,x_0)} (u)$ with
$u=2iK'+v$ and $v\in\mathbb{R}$ are minimal, which provides a proof of
Theorem~\ref{thm:main-3}. This implies in particular that
these exponential functions form the minimal Martin boundary
of the associated killed random walk.

Recall that any isoradial graph can be viewed as a step surface in a hypercubic
lattice $\mathbb{Z}^d$, where $d$
 represents the
number of possible orientations for the unit vectors representing edges of the
diamond graph, and that the exponential function is naturally extended to
$\mathbb{Z}^d$; see Sections~\ref{sec:iso_def} and~\ref{subsec:monotone}.

To each point $\pps=\sum_j n_j e_j$ in the $L^1$-unit ball of
$\mathbb{R}^d$, we associate a probability measure
$\nu_{\pps}$ on the
circle $\mathbb{R}/4K\mathbb{Z}$:
\begin{equation*}
  \nu_{\pps} = \sum_{j=1}^d n_j^+ \delta_{\alpha_j} +
    n_j^-\delta_{\alpha_j+2K},
\end{equation*}
where $n_j^\pm=\max\{\pm n_j,0\}$
represents the asymptotic proportion of steps
$\pm e_j$ in a minimal path from a reference vertex $x$ to $y$, as $y$ tends to
infinity in the direction $\pps$.
Because of the monotonicity of the surface, $n_j^+$ and $n_j^-$ cannot be both
strictly positive. Therefore,
the limiting measures $\nu$ have a support included in half of the
circle: there is an interval of length not smaller than $2K$ with $\nu$-measure
$0$ (see also Section \ref{subsec:ex_as_ex}).

Fix $u=2iK'+v$ with $v$ real.
An important quantity of interest is the rate of growth of the massive
exponential function in the asymptotic direction $\pps$:
\begin{equation*}
  \tau(\pps,u)=\lim_{y\to\pps\infty} \frac{1}{N}\log \vert
  \expo_{(x_0,y)}(u)\vert = \int_{\mathbb{R}/4K\mathbb{Z}}
  \log\left\vert\sqrt{k'}\sc\Bigl(\frac{u-\alpha}{2}\Bigr)\right\vert\nu_{\pps}(\ud\alpha)=\chi(v),
\end{equation*}
see~\eqref{eq:def_chi}.
Because $\sqrt{k'}\sc(u+K)=\frac{1}{\sqrt{k'}\sc(u)}$, the function $\tau$
can be rewritten as
\begin{equation*}
  \tau(\pps,u)
  = \int_{\mathbb{R}/4K\mathbb{Z}}
  \log\left\vert\sqrt{k'}\sc\Bigl(\frac{u-\alpha}{2}\Bigr)\right\vert{\tilde\nu}_{\pps}(\ud\alpha),
\end{equation*}
where $\tilde{\nu}_{\pps} = \sum_{j=1}^d n_j \delta_{\alpha_j}$ is
now a signed measure supported on the angles associated to the basis vectors
$e_1,\ldots, e_d$. It is then clear that
$\tau(-\pps,u)=-\tau(\pps,u)$. Even more, if we extend
by homogeneity the definition of $\tilde{\nu}$ to all vectors of $\mathbb{R}^d$
by $\tilde{\nu}_{\lambda\pps} = \lambda\tilde{\nu}_{\pps}$ for
$\lambda\in\mathbb{R}$,
then $\tau(\cdot,u)$ becomes a non-degenerate linear form, whose kernel is a
hyperplane separating the unit sphere in two parts; see Figure \ref{fig:sphere_inter} for an example.

\begin{figure}
  \centering
  \includegraphics[width=7cm]{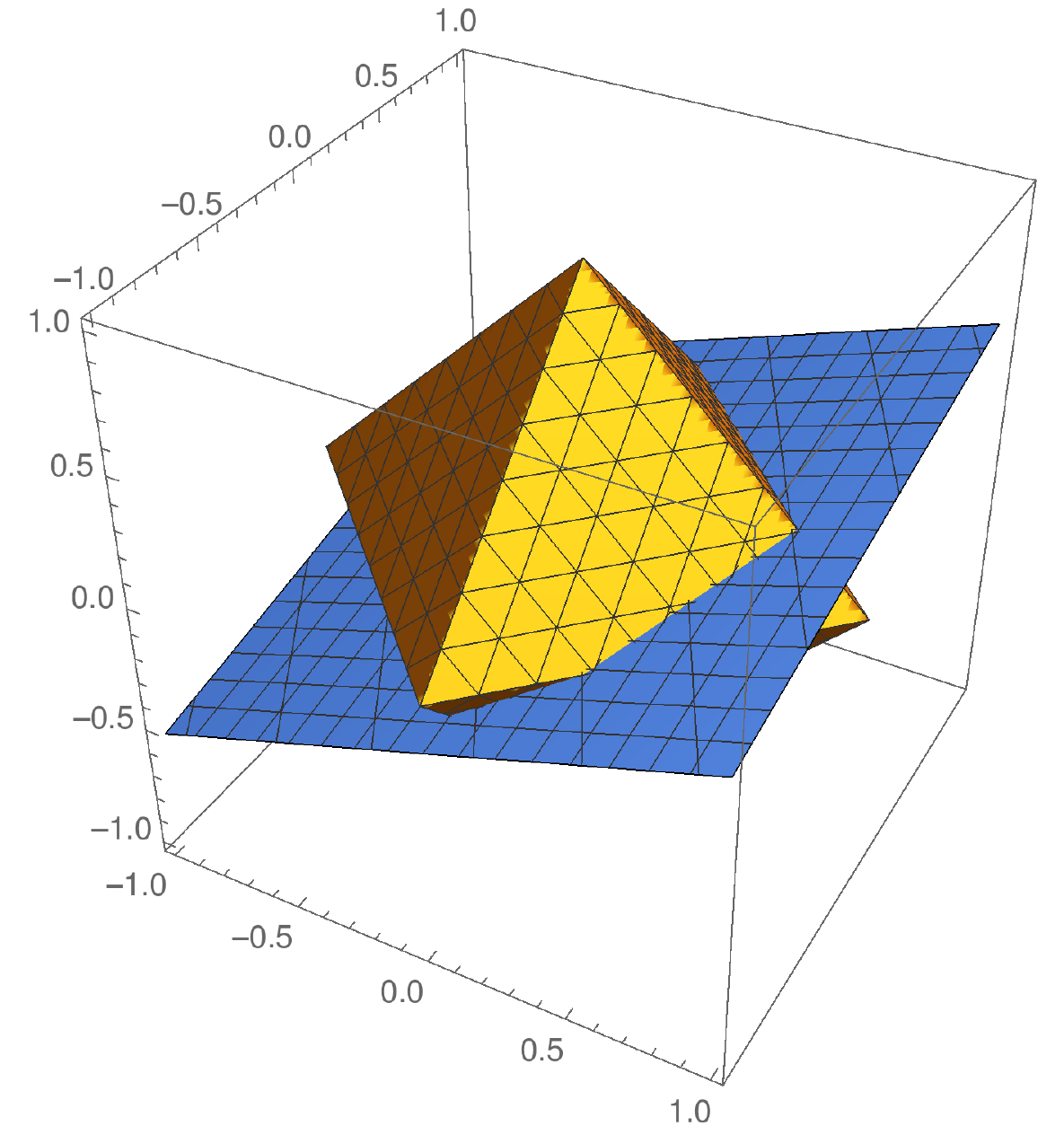}
  \caption{The unit sphere intersected by the kernel of $\tau$ in dimension $3$.}
  \label{fig:sphere_inter}
\end{figure}

\begin{lem}
  The $L^1$-unit sphere is separated in two (connected) hemispheres, 
  corresponding to positive and negative values of $\tau(\cdot,u)$.
\end{lem}
  Informally, this lemma says that for any $u$ in the torus, $\expo_{(x_0,y)}$ tends
  exponentially fast to $0$ for half of the directions to infinity. This is in
  particular true for the positive exponential functions for which $u=2iK'+v$,
  with $v\in\mathbb{R}$.
We claim that this property is characterizing the massive exponential functions
for $u=2iK'+v$ with $v$ real, among all positive massive harmonic functions:

\begin{prop}
  \label{prop:extr_harmo}
  If $h$ is a positive harmonic function which tends to $0$ in at
  least half of the directions, then it is proportional to an exponential
  function.
\end{prop}

\begin{proof}
  Write $h$ as an integral over the positive exponential function for some
  positive measure $\mu$ as in~\eqref{eq:f_expo}. Therefore, the set of directions along which $h$ goes to
  infinity is the union of the half-spaces for which $\expo_{(x_0,y)}(u)$ goes to
  infinity, for $u$ in the support of $\mu$. If $\mu$ is non-trivial then this
  is strictly more than a half-space, which is impossible by hypothesis.
  Therefore $\mu$ should have a support reduced to a singleton, and $h$ is
  proportional to an exponential function.
\end{proof}

\begin{proof}[Proof of Theorem \ref{thm:main-3}]
  If $h$ is less or equal to $h_v$, then it goes to 0 at least in half of the
  directions (the ones where $h_v$ itself goes to 0). We can therefore apply
  Proposition~\ref{prop:extr_harmo} to conclude.
\end{proof}

\begin{figure}
  \centering
  \includegraphics[width=12cm]{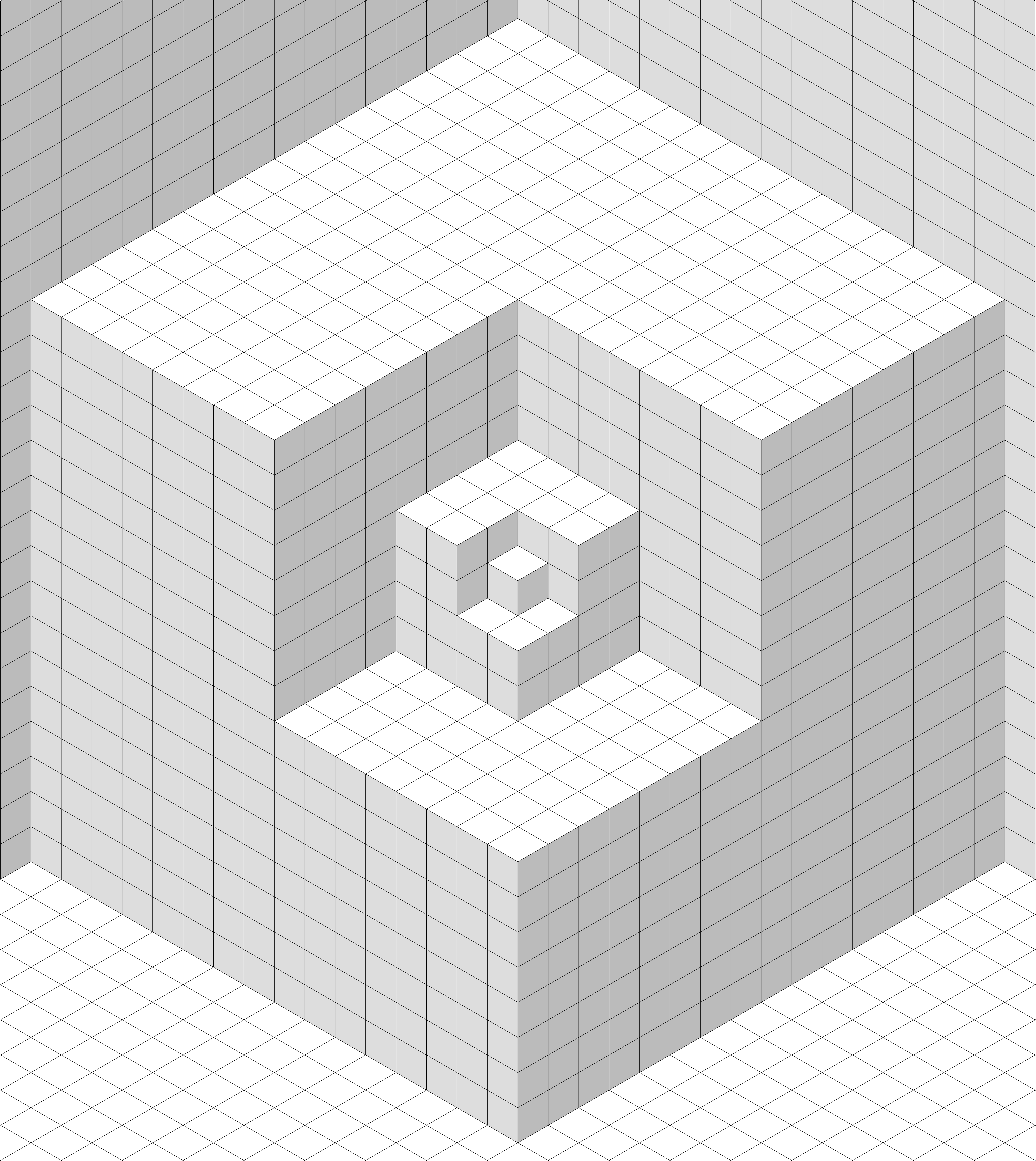}
  \caption{A piece of the diamond graph of an isoradial graph which is not
    asymptotically flat. From the center, there are ``waves'' of size growing
    exponentially. When going to infinity in a direction in the plane, the
    proportions of the different types of steps do not converge, but instead
  oscillate.}
  \label{fig:waves}
\end{figure}

\appendix
\section{Some thoughts on a positivity question (by Alin Bostan)}
\label{sec:app_AB}

We consider the following question: let $s(k)=\sum_{n \geq  0} s_n k^n$, 
starting
\begin{equation*}     s(k) =\frac{1}{2}+\frac{3}{32}k^2+\frac{3}{64}k^4+\frac{123}{4096}k^6+\frac{177}{8192}k^8+\frac{34887}{2097152}k^{10}+\cdots,
\end{equation*}
be the unique power series solution of the algebraic equation
\begin{equation}
\label{eq:alg-s}
     k^2s^4-2k^2s^3+2s-1=0.
\end{equation}
We want to prove that all the coefficients of $s(k)$ are non-negative,
and more precisely that all the coefficients $s_{2n}$ are positive.

Note that the sequence $(s_n)_{n \geq  0}$ satisfies various types of
recurrences, either non-linear,
obtained by extracting the coefficient of $k^n$
in both sides of~\eqref{eq:alg-s}, or even linear, obtained by using a linear
differential equation (with polynomial coefficients in~$k$) satisfied by
$s(k)$. However, these recurrences cannot be used directly to solve our
positivity question. For instance, from the algebraic
equation~\eqref{eq:alg-s}, Cockle's algorithm (see~\cite{Issac} and the
references therein) proves that $s(k)$ satisfies the 3rd order linear inhomogeneous differential equation
\begin{multline}
9\,{k}^{3} \left( k-1 \right) ^{2} \left( k+1 \right) ^{2}
s''' \left( k \right)
+9\,{k}^{2} \left( k-1 \right)  \left( k+1 \right)  \left( 5\,{k}^{2}-1 \right) s''
 \left( k \right) \label{eq:diffeqS} \\
+k \left( 35\,{k}^{4}-14\,{k}^{2}-13 \right) s'
 \left( k \right) 
-4\, \left( {k}^{2}-2 \right) s \left( k \right) 
= 4-2\,{k}^{2}, 
\end{multline}
from which a coefficient extraction proves that $(s_n)_{n \geq  0}$ satisfies the linear recurrence 
\begin{multline}
	\left( 3\,n+10 \right)  \left( 3\,n+14 \right)  \left( n+2 \right) s_{n+4}
+
n \left( 3\,n+2 \right)  \left( 3\,n+4 \right) s_n  \label{eq:recS}\\
= 2\, \left( 9\,{n}^{3}+54\,{n}^{2}+106\,n+70 \right) s_{n+2}. 
\end{multline}
It is not obvious from this recurrence how to derive a 
non-negativity proof
because of the plus sign in front of the coefficient of $s_n$. In principle,
two quite general methods might be applied to this kind of question (alone, or
combined): one relying on (effective) asymptotics for the coefficients of
algebraic functions~\cite[p.~504--505]{FlSe09}, the other relying on the
approach of~\cite[\S 4.2]{KP-10}. Instead, we present two different proofs.

The first one is based on the following 
 hypergeometric expression for $s(k)$:
\begin{equation}\label{eq:closed}
 s(k)=\frac{1}{2}+  \frac{3}{2} \cdot 
\left(
\frac{k}{4} \cdot
\pFq{2}{1}{\frac12,\frac56}{\frac53}{k^2}  \right)^2,
\end{equation}
where the Gauss' hypergeometric series~$\pFqcomma$ with parameters 
$\frac12, \frac56$ and $\frac53$ is defined by
\begin{equation}\label{eq:2F1}
\pFq{2}{1}{\frac12,\frac56}{\frac53}{k} = \sum_{n=0}^\infty \frac{(\frac12)_n(\frac56)_n}{(\frac53)_n} \, \frac
{k^n} {n!}
= 1+{\frac{1}{4}}k+{\frac{33}{256}}{k}^{2}+{\frac{85}{1024}}{k}^{3}
+\cdots
,
\end{equation}
and $(x)_n$ denotes the Pochhammer symbol $(x)_n=x(x+1)\cdots(x+n-1)$ for $n\in\mathbb{N}$.

Once found (e.g., using algorithmic tools for automated
guessing~\cite{Kauers18}, or for differential equation solving~\cite{IH-17}),
equality~\eqref{eq:closed} can easily be proved using closure properties of
D-finite functions~\cite{Stanley80}: from the 2nd order linear differential
equation satisfied by $\pFq{2}{1}{\frac12,\frac56}{\frac53}{k}$, one computes
a 3rd order differential equation satisfied by the right-hand side
of~\eqref{eq:closed}, which appears to coincide with the differential
equation~\eqref{eq:diffeqS} satisfied by the left-hand side
of~\eqref{eq:closed}. Therefore, the coefficients sequences of both right-hand
side and left-hand side of~\eqref{eq:closed} satisfy the recurrence
relation~\eqref{eq:recS}, hence they are equal since their initial terms
coincide.

A shorter proof of equality~\eqref{eq:closed}, with a different (more
geometric) flavor, goes as follows: on the one hand, the hypergeometric
function in~\eqref{eq:2F1} satisfies the algebraic
transformation\footnote{This is inspired by the Darboux covering for
tetrahedral hypergeometric equations of the Schwarz type $(1/3, 1/3,
2/3)$~\cite[\S6.1]{Vidunas08}, see also~\cite{Vidunas09}. Note that the same
change of variables has been used in~\cite[\S4]{BKR17}.}
\[
\pFq{2}{1}{\frac12,\frac56}{\frac53}{{\frac {x \left( x+2 \right) ^{3}}{ \left( 2\,x+1 \right) ^{3}}}}
=4\,{\frac { \left( 2\,x+1 \right) ^{\frac32}}{ \left( x+2 \right) ^{2}}}.\]
On the other hand, it is easy to prove, starting from the polynomial~\eqref{eq:alg-s}, that the algebraic function $s(k)$ satisfies
\[
s \left( {\frac {x^\frac12 \left( x+2 \right) ^{\frac32}}{ \left( 2\,x+1 \right) ^{\frac32}}} \right) = {\frac {2\,x+1}{x+2}}.\]
Denoting $\displaystyle{\varphi(x) = { {x \left( x+2 \right) ^{3}} / { \left( 2\,x+1 \right) ^{3}}}}$,
identity~\eqref{eq:closed} becomes trivial in the new variable~$x$:
\[{\frac {2\,x+1}{x+2}}=\frac12+\frac{3}{2} \cdot \varphi(x)
 \cdot \left( {\frac { \left( 2\,x+1 \right) ^{3/2}}{ \left( x+2 \right) ^{2}}}\right)^2.\]
Now that~\eqref{eq:closed} is proved, the positivity of the coefficients 
$s_{2n}$ of $s(k)$ is a direct consequence of the obvious fact that the 
hypergeometric power series in~\eqref{eq:2F1} has positive coefficients.

\smallskip The second proof is based on the so-called \emph{Stieltjes
inversion formula}~\cite[Lecture~2]{NiSp06}, that allows to express in some
cases the $n$-th term of a sequence as the $n$-th moment of a positive
density. Using this formula, one can deduce the following expression:
\[
s_{2n}=
\frac{\sqrt{3}}{\pi}\cdot
\int_{0}^1
k^n\cdot
H_1(k) \cdot
\left(
\frac{1}{{2}^{4/3}}\cdot
{\frac {H_2(k)}{{k}^{2/3}}}
-\frac{1}{{2}^{8/3}}
\cdot
\frac {H_1(k)}{{k}^{1/3}}
\right)\ud k,
\]
where $H_1(k)$ and $H_2(k)$ are the hypergeometric functions
\[H_1 = \pFq{2}{1}{\frac16,\frac56}{\frac43}{k} = 1+{\frac{5}{48}}k+\cdots, 
\quad H_2 = \pFq{2}{1}{-\frac16,\frac12}{\frac23}{k} 
=1-{\frac{1}{8}}k-{\frac{3}{64}}{k}^{2}-\cdots.
\]
At this point, positivity is not yet apparent. However, using the change of variables $k=\varphi(x)$, and using the fact that $H_1(\varphi(x))$ and $H_2(\varphi(x))$ become the simple algebraic functions
$2\,{ {{(x+1)^{1/3}} {(2\,x+1)^{1/2}}}/{(x+2)}}$ and $(2\,x+1)^{-1/2}$, the previous expression simplifies to:
\begin{equation}\label{eq:closed2}
	s_{2n} =  \frac{\sqrt {3} \cdot \sqrt [3]{2}}{\pi} \cdot \int_{0}^1 {\frac { \left( x-1 \right) ^{2} \left( \sqrt [3]{2}-\sqrt [3]{x^2+x} \right) \sqrt [3]{x+1}}{ \left( x+2
 \right)  \left( 2\,x+1 \right) ^{2}{x}^{2/3}}}  \varphi(x)^n\ud x.
\end{equation}
Once found, equality~\eqref{eq:closed2} can again be proved automatically using algorithmic
tools, notably the method of ``creative
telescoping''~\cite{AlZe90,Koutschan13,BDS16}.

Note that Lagrange inversion is a tempting alternative to the 
previous moment approach; however, in our case, it gives that for all $n>0$,
\[
s_{2n} = \frac{1}{2n} \cdot [z^{-1}] \left( {\frac { \left( z+1 \right) ^{3} \left( 3-z \right) }{16\, z}} \right) ^{n}
= \frac{(-1)^n}{{n\cdot {2}^{4n+1}}} \cdot \sum _{i=0}^{n}{3\,n\choose i}{n\choose i+1} \left( -3 \right) ^{i+1},
\]
and positivity is not clear on any of these expressions.

In contrast, identity~\eqref{eq:closed2} obviously proves that the sequence
$(s_{2n})_{n\geq 0}$ is positive. And it actually proves much more, namely
that $(s_{2n})_{n\geq 0}$ is a \emph{Stieltjes-Hausdorff moment sequence}, and
in particular that it is log-convex, i.e., $ \, s_{2n+2}s_{2n-2} \, \geq
s_{2n}^2$ for all $n\geq 1$~\cite[\S2]{Sokal} and $\Delta^k s_{2n} \geq 0$ for
all $k\geq 0$~\cite[p.~338]{Sp-64}, where $\Delta$ is the difference operator
$\Delta (c_n) = (c_n - c_{n+1})$.

\medskip
We also consider the following related question: let $t(k)=\sum_{n \geq  0} t_n k^n$, starting
\begin{equation*}
t(k)=1+\frac{3}{64}k^4+\frac{3}{64}k^6+\frac{711}{16384}k^8+\frac{327}{8192}k^{10}+\cdots,
\end{equation*}
be the unique power series solution of the algebraic equation
\begin{equation} \label{eq:alg-t}
     -27(1-k^2)t^4+18(1-k^2)t^2+2(2-k^2)^2t+1-k^2+k^4=0.
\end{equation}
We want to prove that all the coefficients $t_n$ are non-negative.

First, we remark that $t(\varphi(x)^{\frac12}) = { {({x}^{2}+x+1)}/{(1+x-2x^2)}}$. Since 
$s(\varphi(x)^{\frac12}) = (2x+1)/(x+2)$, 
this implies that 
$t(k) = R(s(k))$,
where $R$ is the rational function
\[R(z) = {\frac {{z}^{2}-z+1}{3\, z \left( 1-z \right) }}.
\]
As a consequence of this and of identity~\eqref{eq:closed}, we get that 
\begin{equation} \label{eq:closed-t}
	t(k) = (1+3 H^4)/(1-9 H^4) = (1+3 H^4) \cdot \sum_{n \geq 0} (9H^4)^n, \end{equation}
where $H$ is the hypergeometric function with non-negative coefficients
\[ H = \frac{k}{4} \cdot
\pFq{2}{1}{\frac12,\frac56}{\frac53}{k^2}  
=
 {\frac{1}{4}}k+{\frac{1}{16}}{k}^{3}+{\frac{33}{1024}}{k}^{5}+\cdots,\]
therefore all the coefficients of $t(k)$ are non-negative.

\end{document}

%% file: fig_rhombus_angle.pdf_t
\begin{picture}(0,0)%
\includegraphics{fig_rhombus_angle.pdf}%
\end{picture}%
\setlength{\unitlength}{4144sp}%
\begingroup\makeatletter\ifx\SetFigFont\undefined%
\gdef\SetFigFont#1#2#3#4#5{%
  \reset@font\fontsize{#1}{#2pt}%
  \fontfamily{#3}\fontseries{#4}\fontshape{#5}%
  \selectfont}%
\fi\endgroup%
\begin{picture}(1380,940)(4711,-3869)
\put(5292,-3543){\makebox(0,0)[lb]{\smash{{\SetFigFont{12}{14.4}{\rmdefault}{\mddefault}{\updefault}{\color[rgb]{0,0,0}$e$}%
}}}}
\put(4861,-3796){\makebox(0,0)[lb]{\smash{{\SetFigFont{12}{14.4}{\rmdefault}{\mddefault}{\updefault}{\color[rgb]{0,0,0}$e^ {i\overline{\alpha}_e}$}%
}}}}
\put(4816,-3166){\makebox(0,0)[lb]{\smash{{\SetFigFont{12}{14.4}{\rmdefault}{\mddefault}{\updefault}{\color[rgb]{0,0,0}$e^ {i\overline{\beta}_e}$}%
}}}}
\put(5071,-3327){\makebox(0,0)[lb]{\smash{{\SetFigFont{12}{14.4}{\rmdefault}{\mddefault}{\updefault}{\color[rgb]{0,0,0}$\overline{\theta}_e$}%
}}}}
\end{picture}%